\documentclass[11pt,reqno]{amsart}

\usepackage{amsmath,amssymb,amsthm,bbm,bm}
\usepackage{color}
\usepackage[colorlinks=true, allcolors=blue]{hyperref}
\usepackage{mathrsfs}
\usepackage{mathtools}

\usepackage[noabbrev,capitalize,nameinlink]{cleveref}
\crefname{equation}{}{}
\usepackage[noadjust]{cite}

\usepackage{fullpage}

\usepackage{graphics}
\usepackage{pifont}
\usepackage[T1]{fontenc}
\usepackage{tikz}
\usetikzlibrary{arrows.meta}
\usepackage{environ}
\usepackage{framed}
\usepackage{url}
\usepackage[linesnumbered,ruled,vlined]{algorithm2e}
\usepackage[noend]{algpseudocode}
\usepackage[labelfont=bf]{caption}
\usepackage{cite}
\usepackage{framed}
\usepackage[framemethod=tikz]{mdframed}
\usepackage{appendix}
\usepackage{graphicx}
\usepackage[textsize=tiny]{todonotes}
\usepackage{tcolorbox}
\usepackage{enumerate}
\allowdisplaybreaks[1]
\usepackage{enumerate}

\crefname{algocf}{Algorithm}{Algorithms}

\crefname{equation}{}{} %remove ``Equation''
\AtBeginEnvironment{appendices}{\crefalias{section}{appendix}} %appendices

\usepackage[color,final]{showkeys} %add in 'final' into parameter to remove showkeys

\colorlet{refkey}{orange!20}
\colorlet{labelkey}{blue!30}

\crefname{algocf}{Algorithm}{Algorithms}

\numberwithin{equation}{section}
\newtheorem{theorem}{Theorem}[section]
\newtheorem{proposition}[theorem]{Proposition}
\newtheorem{lemma}[theorem]{Lemma}
\newtheorem{claim}[theorem]{Claim}
\crefname{claim}{Claim}{Claims}

\newtheorem*{question*}{Question}

\theoremstyle{definition}
\newtheorem{definition}[theorem]{Definition}

\newtheorem*{definition*}{Definition}

\theoremstyle{remark}
\newtheorem*{remark}{Remark}

\newcommand{\snorm}[1]{\lVert#1\rVert}
\newcommand{\sang}[1]{\langle #1 \rangle}

\newcommand{\mb}{\mathbb}
\newcommand{\mbf}{\mathbf}
\newcommand{\mbm}{\mathbbm}
\newcommand{\mc}{\mathcal}
\newcommand{\mf}{\mathfrak}

\newcommand{\msf}{\mathsf}
\newcommand{\ol}{\overline}
\newcommand{\on}{\operatorname}
\newcommand{\wh}{\widehat}

\newcommand{\E}{\mathbb{E}}
\newcommand{\F}{\mathbb{F}}

\newcommand{\fB}{\mathfrak{B}}

\newcommand{\sB}{\mathsf{B}}

\newcommand{\TT}{\intercal}

\allowdisplaybreaks

\title{Popular differences for matrix patterns}

\author[Berger]{Aaron Berger}
\author[Sah]{Ashwin Sah}
\author[Sawhney]{Mehtaab Sawhney}
\author[Tidor]{Jonathan Tidor}
\address{Department of Mathematics, Massachusetts Institute of Technology, Cambridge, MA 02139, USA}
\thanks{Berger, Sah, Sawhney, and Tidor were supported by NSF Graduate Research Fellowship Program DGE-1745302.}
\email{\{bergera,asah,msawhney,jtidor\}@mit.edu}

\begin{document}

\begin{abstract}
The following combinatorial conjecture arises naturally from recent ergodic-theoretic work of Ackelsberg, Bergelson, and Best. Let $M_1$, $M_2$ be $k\times k$ integer matrices, $G$ be a finite abelian group of order $N$, and $A\subseteq G^k$ with $|A|\ge\alpha N^k$. If $M_1$, $M_2$, $M_1-M_2$, and $M_1+M_2$ are automorphisms of $G^k$, is it true that there exists a popular difference $d \in G^k\setminus\{0\}$ such that
\[\#\{x \in G^k: x, x+M_1d, x+M_2d, x+(M_1+M_2)d \in A\} \ge (\alpha^4-o(1))N^k.\]
We show that this conjecture is false in general, but holds for $G = \mathbb{F}_p^n$ with $p$ an odd prime given the additional spectral condition that no pair of eigenvalues of $M_1M_2^{-1}$ (over $\overline{\mathbb{F}}_p$) are negatives of each other. In particular, the ``rotated squares'' pattern does not satisfy this eigenvalue condition, and we give a construction of a set of positive density in $(\mathbb{F}_5^n)^2$ for which that pattern has no nonzero popular difference. This is in surprising contrast to three-point patterns, which we handle over all compact abelian groups and which do not require an additional spectral condition.
\end{abstract}

\maketitle

\section{Introduction}\label{sec:intro}
\subsection{Popular patterns and past results}\label{sub:past-results}
Using an argument of Varnavides \cite{Var59}, it is well-known that Roth's theorem \cite{Rot53} on three-term arithmetic progressions can be strengthened to guarantee at least $c_\alpha N^2$ arithmetic progressions in a set $A\subseteq[N]$ of size $\alpha N$. The constant $c_{\alpha}$ is known not to be polynomial in $\alpha$; in particular, modifying a well-known construction of Behrend \cite{Beh46} allows one to construct sets with $\alpha^{c\log(1/\alpha)}N^2$ three-term arithmetic progressions. However, Green \cite{Gre05}, showed that one has a ``popular'' common difference $d\neq 0$, i.e., a value $d\in[N]$ such that
\[\#\{a: a,a+d,a+2d\in A\}\ge(\alpha^3-o(1))N.\]
That is, the set behaves like a random set along certain structured differences, if not all of them. Green's proof involves an arithmetic regularity lemma, which is essentially equivalent to arithmetic regularity for the Gowers $U^2$-norm.

One can ask if this phenomenon holds for longer arithmetic progressions. The analogous result for four-term arithmetic progressions with $\alpha^4-o(1)$ on the right-hand side was proved by \cite{GT10} relying on a remarkable ``positivity'' identity \cite{GT10} (see \cite{Gre06} for a version over $\mb{F}_p^n$ for $p\ge 5$) in combination with the $U^3$-arithmetic regularity results of Green and Tao \cite{GT10}. However, surprisingly, an $\alpha^k-o(1)$ (or any polynomial) bound does not hold for $k$-term arithmetic progressions for $k\ge 5$ due to a construction of Ruzsa \cite[Appendix]{BHKR05}. These results were motivated by corresponding ergodic results of Bergelson, Host, and Kra \cite{BHKR05}, although the theorems do not directly transfer when studying popular differences (as opposed to Furstenberg's correspondence theorem for Szemer\'edi's theorem).

One may ask about popularity of more general patterns, for example $\{0,1,2,4\}$. (We use a set to refer to the pattern consisting of homothetic copies of that set; in this case, the pattern is $(a,a+d,a+2d,a+4d)$.) The proof of Green and Tao \cite{GT10} for four-term arithmetic progressions (and \cite{Gre06} over finite fields) immediately extends to patterns of the form $\{0,k_1,k_2,k_1+k_2\}$ for $k_1k_2(k_1+k_2)\neq 0$. Work of the second and third authors and Zhao \cite{SSZ20} shows that two-point, three-point, and these specific ``parallelogram'' four-point patterns are the only popular patterns over $\mb{Z}$.

Popularity of higher-dimensional patterns such as corners, $\{(0,0),(1,0),(0,1)\}$, was first studied by Mandache \cite{Man18} in the combinatorial setting (see \cite{Chu11,DS16} for related work in the ergodic theory setting), who showed over $\mb{F}_p^n$ that they are not $\alpha^3$-popular but do satisfy a weakened bound with $\alpha^4$ instead. Fox, the second and third authors, Stoner, and Zhao \cite{FSSSZ} showed that the tight bound is of the form $\alpha^4\tau(\alpha)$, where $\tau$ grows as $\alpha\to 0$, but is of the form $\alpha^{o(1)}$. Finally, the first author \cite{Ber19} showed the same behavior over $\mb{Z}^2$. The second and third authors and Zhao \cite{SSZ20} studied higher-dimensional patterns which are homothetic copies of a set and provide a nearly comprehensive classification.

\subsection{Our contributions}\label{sub:our-contrib}
The standard toolset of arithmetic regularity in higher-order Fourier analysis, which can prove popular difference results for three and four-point single-dimensional patterns, necessarily breaks to some extent when handling higher-dimensional corners (as pointed out in \cite{SSZ20}), and has not yet been successfully applied to four-point patterns such as squares for which the question remains open. However, it was noted by Prendiville \cite{Pre15} that classic single-dimensional techniques extend if one considers \emph{full-rank} matrix patterns (a collection which excludes corners and squares but includes a wide class of multidimensional configurations such as ``rotated corners'' -- also known as ``right isosceles triangles'' -- and ``rotated squares''), and he achieves versions of Szemer\'edi's theorem (for $k\le 4$ points) with good quantitative bounds in this setting. We continue in this line of work, achieving popular difference results of strength equal to the single-dimensional case, illustrating by comparison the suitability of these methods to full-rank patterns.

The main novelty of this paper lies in the popular difference results for four-point patterns, where we exhibit further behavior that does not appear even in Prendiville's work. In order to properly handle popularity of four-point patterns, we show that one must apply the method of arithmetic regularity in a manner that sees the spectral properties of the matrices defining the pattern. In particular, the counting lemma (which for four-point patterns over $\mb{F}_p^n$ relies heavily on equidistribution over parts in quadratic factors) becomes qualitatively distinct depending on the spectral structure of the matrices in the pattern (see \cref{thm:counting}). This subtlety is not present in earlier counting lemmas for scalar-valued patterns. This also translates concretely to an additional restriction that no pair of eigenvalues of an associated matrix can be negatives of each other for our method to produce a popular difference result (see \cref{thm:main}). To confirm that this behavior is genuine and not an artifact of the proof, in \cref{thm:counterexample} we exhibit a full-rank matrix pattern which does not satisfy the additional spectral condition imposed by \cref{thm:main} and for which the conclusion of the theorem is false. In particular, we show that rotated squares in $\mb{F}_5^n$ do not satisfy a popular difference result, at least with popularity $\alpha^4$.

\subsection{Summary of results}\label{sub:summary-results}
We first prove a popular differences result for all full-rank three-point patterns. A three-point pattern is full rank if it can be expressed in the form $\vec{x},\vec{x}+M_1\vec{d},\vec{x}+M_2\vec{d}$ where $M_1,M_2,M_1-M_2$ are invertible. One such example is ``rotated corners,'' which are of the form $(x,y),(x+a,y+b),(x+b,y-a)$. (By contrast, standard corners $(x,y),(x+a,y),(x,y+a)$ are not full rank.) As a special case, this resolves a conjecture of Ackelsberg, Bergelson, and Best \cite[Question 1.21]{ABB21}, which concerns the case of rotated corners specifically. Kova\v{c} \cite{Kov21} has independently proved this rotated corners conjecture with similar methods. 

\begin{theorem}\label{thm:3ptsZ}
Let $M_1, M_2$ be $k \times k$ invertible integer matrices so that $M_1-M_2$ is invertible. For any $\alpha,\epsilon > 0 $ there exists $N_0(\alpha,\epsilon,M_1,M_2)$ so that the following holds.  If $N \ge N_0$, then for any $A \subseteq [N]^k$, $|A| \ge \alpha N^k$, there is a popular difference $\vec d \neq 0$ so that
\begin{equation*}
\#\{\vec x \in [N]^k : \vec x, \vec x+M_1\vec d, \vec x+M_2\vec d \in A\}\ge (\alpha^3 - \epsilon) N^k.
\end{equation*}
\end{theorem}

We additionally prove an analogous version of the result where the interval $[N]$ is replaced by an arbitrary compact abelian group $G$. See \cref{sec:three-point} for the precise statement and proof of this result.

We turn next to four-point patterns of matrices. There are a few natural restrictions on generic patterns $\vec x, \vec x+M_1\vec d, \vec x+M_2\vec d, \vec x+M_3\vec d$ that arise when trying to prove a popular differences result. First, we impose $M_3 = M_1 + M_2$, which is a generalization of the ``parallelogram'' condition in the popular differences result of Green and Tao \cite{GT10}. Second, we require that $M_1, M_2, M_1 - M_2, M_1+M_2$ are all invertible; in this case we call the pattern full rank.\footnote{
Axis-aligned squares are an example of a four-point pattern that is not full rank, and for which the version of popular differences we would like to prove is known to be false; see \cite[Theorem~3.1]{SSZ20}.} 
The combination of these two conditions is analogous to the ``admissibility'' condition of \cite{ABB21}, and essentially appears in \cite{Pre15}. One might guess that they are sufficient to guarantee popular differences. In this paper we show that this guess is incorrect by demonstrating  the necessity of an additional spectral condition on the pattern. In the spirit of the finite field philosophy advocated by Green \cite{Gre05b}, we restrict attention to the finite field model $G=\mb{F}_p^n$ with $p$ an odd prime. We suspect our methods can be extended to handle more general abelian groups, but choose to avoid the complexity of the inverse theorems for the $U^3$-norm over general abelian groups.

\begin{theorem}\label{thm:main}
Fix $k\ge 1$ and $p$ an odd prime. Let $M_1,M_2$ be $k\times k$ matrices with coefficients in $\mb{F}_p$ such that $M_1$, $M_2$, $M_1-M_2$, and $M_1+M_2$ are invertible and no pair of eigenvalues of $M_1M_2^{-1}$ (viewed over $\overline{\mb{F}}_p$) are negatives of each other. For $\alpha,\epsilon > 0$, there exists $n_0(\alpha,\epsilon,p)$ such that the following holds. If $n\ge n_0$, then for any $A\subseteq (\mb{F}_p^n)^k$, $|A|\ge\alpha p^{nk}$, there is a popular difference $\vec d\neq 0$ so that
\begin{equation*}
\#\{\vec x \in (\mb{F}_p^n)^k : \vec x, \vec x+M_1\vec d, \vec x+M_2\vec d, \vec x+(M_1+M_2)\vec d \in A\}\ge (\alpha^4 - \epsilon) p^{nk}.
\end{equation*}
In fact, there are $\Omega_{\alpha,\epsilon,p}(p^{nk})$ values of $\vec d$ that work.\end{theorem}

Furthermore, we show that one cannot completely remove the spectral condition.
\begin{theorem}\label{thm:counterexample}
There is an absolute constant $c > 0$ such that the following holds. If $\alpha\in(0,c)$, then for all sufficiently large $n$ (depending on $\alpha$) there is a set $A\subseteq(\mb{F}_5^n)^2$ satisfying $|A|\ge\alpha 5^{2n}$ and
\[\max_{(a,b)\neq 0}\#\{(x,y): (x,y),(x+a,y+b),(x+b,y-a),(x+a+b,y+b-a)\in A\}\le (1-c)\alpha^45^{2n}.\]
\end{theorem}
Here the associated matrices are
\[M_1 = \begin{bmatrix}1&0\\0&1\end{bmatrix},\quad M_2 = \begin{bmatrix}0&-1\\1&0\end{bmatrix}.\]
Note that the eigenvalues of $M_1M_2^{-1}$ indeed are negatives of each other. We believe it is likely that one can construct a counterexample for all $k\times k$ matrices with some pair of negated eigenvalues by lifting the ideas involved in this construction.

Although there is no direct implication, this can be seen as a combinatorial finite field analogue of \cite[Question~1.11]{ABB21} and we expect our counterexample can be extended to the ergodic setting. In particular, we answer the combinatorial analogue in the negative but point to a potential new condition under which their question might be resolved positively.

\subsection{Notation and outline}\label{sub:outline}
We use $O,o,\Omega$ as standard asymptotic notation. Subscripts in said notation denote dependence of the implicit constants on those subscripts.

The majority of this paper, \cref{sec:gowers-norms,sec:arith-reg,sec:equidistribution,sec:4-point-popular-differences}, is devoted to the proof of \cref{thm:main}, the popular difference result for four-point patterns. See \cref{sec:gowers-norms} for an outline of that argument. In \cref{sec:counterexample}, we construct the counterexample that proves \cref{thm:counterexample}. Finally, in \cref{sec:three-point}, we show \cref{thm:3ptsZ}, the three-point pattern result.

\subsection*{Acknowledgements}
We thank our advisor Yufei Zhao for introducing us to the study of popular differences in additive combinatorics.

\section{Gowers norms and arithmetic regularity}
\label{sec:gowers-norms}

The proof of \cref{thm:main} proceeds in three steps, following the now-standard framework of the arithmetic regularity method.

\smallskip

First, we show that the matrix patterns we are interested in are controlled by an appropriate Gowers $U^s$-norm. Results of this nature are sometimes referred to as ``generalized von Neumann theorems''. The definition of the Gowers norms and the proof of this result are given in this section.

\smallskip

Second, we prove an arithmetic regularity lemma, which gives a decomposition of an arbitrary function $f\colon G\to\mb{C}$ as $f=f_{\on{str}}+f_{\on{sml}}+f_{\on{psr}}$ into a ``structured'', ``small'', and ``pseudorandom'' piece. For our application in groups $G^k$, it will be necessary to carefully define ``structured'' in a way that is adapted to the product structure of $G^k$. This definition and the proof of this result is given in \cref{sec:arith-reg}.

\smallskip

Third, we prove novel equidistribution results in order to understand the counts of matrix patterns inside the structured piece, $f_{\on{str}}$. These results occur in \cref{sec:equidistribution}. Combining these three steps, we prove \cref{thm:main} in \cref{sec:4-point-popular-differences}.

\begin{definition}\label{def:gowers-norm}
Fix an integer $s\ge 1$ and a finite abelian group $G$. For a function $f\colon G\to\mb{C}$, the \emph{Gowers $U^s$-norm} is defined by
\[\snorm{f}_{U^s(G)} = \left(\mb{E}_{x,h_1,\ldots,h_s\in G}\prod_{\omega\in\{0,1\}^s}\mc{C}^{|\omega|}f(x+\omega_1h_1 + \cdots + \omega_sh_s)\right)^{1/2^s},\]
where $\mc{C}$ denotes the complex conjugation operator and $|\omega| = \omega_1 + \cdots + \omega_s$.
\end{definition}

It is well-known that the above is indeed a norm when $s\ge 2$. (For $s = 1$ it is the seminorm $f\mapsto |\mb{E}_{x\in G}f(x)|$, so the term ``Gowers norm'' is a slight misnomer.) A useful equivalent definition is that \[\|f\|_{U^s(G)}^{2^s}=\mb{E}_{h\in G}\|\partial_h f\|_{U^{s-1}(G)}^{2^{s-1}}\] where the \emph{multiplicative derivative} $\partial_h f$ is defined by $(\partial_h f)(x)=f(x)\ol{f(x+h)}$.

We now prove that full-rank matrix patterns are controlled by an appropriate $U^s$-norm. The typical setup in this paper is to consider a pattern of the form $\vec x+M_1\vec d,\vec x+ M_2\vec d,\ldots,\vec x+ M_s\vec d$ in $G^k$ where $M_1,\ldots, M_s$ are $k\times k$ matrices with certain non-degeneracy conditions. In particular we assume that $M_i$ and $M_i-M_j$ are invertible for each $i\neq j$.

This lemma is true even in the general setting where we replace the matrix $M_i$ acting on $G^k$ by an arbitrary autmorphism $A_i$ acting on $G^k$. In this general setting, the product structure on $G^k$ is no longer important. As the proof of the more general version is no more difficult than the original result, we include it here. The proof follows by an application of the Cauchy--Schwarz inequality;  similar results for specific patterns are implicit in the literature (e.g., \cite{Pre15}).

\begin{lemma}\label{lem:gowers-inequality}
Let $s\ge 2$, and $G$ be a finite abelian group. Let $A_1,\ldots,A_s$ be automorphisms of $G$ such that $A_i-A_j$ is an automorphism for each $i\neq j$. Then for functions $f_i\colon G\to\mb{C}$ satisfying $\snorm{f_i}_\infty\le 1$ we have
\[|\mb{E}_{x,d\in G}f_1(x+A_1d)\cdots f_s(x+A_sd)|\le\min_{i\in[s]}\snorm{f_i}_{U^{s-1}(G)}.\]
\end{lemma}
\begin{proof}
We induct on $s$. For $s=2$, note that 
\begin{align*}
|\mb{E}_{x,d\in G}f_1(x+A_1d)f_2(x+A_2d)| &= |\mb{E}_{x,d\in G}f_1(x)f_2(x+(A_2-A_1)d)| \\
&= |\mb{E}_{x,y\in G}f_1(x)f_2(y)|\\
&=\snorm{f_1}_{U^{1}(G)} \snorm{f_2}_{U^{1}(G)}.
\end{align*}
Since $\snorm{f_i}_\infty\le 1$, the result follows in this case. Now suppose $s\ge 3$. We have
\begin{align*}
|\mb{E}_{x,d\in G}&f_1(x+A_1d)\cdots f_s(x+A_sd)|\\
&= |\mb{E}_{x,d\in G}f_1(x+(A_1-A_s)d)\cdots f_{s-1}(x+(A_{s-1}-A_s)d)f_s(x)|\\
&\le\mb{E}_{x}|\mb{E}_{d}f_1(x+(A_1-A_s)d)\cdots f_{s-1}(x+(A_{s-1}-A_s)d)|\\
&\le\left(\mb{E}_{x}|\mb{E}_{d}f_1(x+(A_1-A_s)d)\cdots f_{s-1}(x+(A_{s-1}-A_s)d)|^2\right)^{1/2}\\
&=\big(\mb{E}_{x}\mb{E}_{d,h}f_1(x+(A_1-A_s)d)\cdots f_{s-1}(x+(A_{s-1}-A_s)d)\\
&\qquad\cdot\ol{f_1}(x+(A_1-A_s)d+(A_1-A_s)h)\cdots\ol{f_{s-1}}(x+(A_{s-1}-A_s)d+(A_{s-1}-A_s)h)\big)^{1/2}.
\end{align*}
To bound the last expression, we apply the induction hypothesis with the maps $A_1-A_s,\ldots,A_{s-1}-A_s$ and the functions $\partial_{(A_i-A_s)h}f_i$. Note that by hypothesis, the maps $A_i-A_s$ are automorphisms as are $(A_i-A_s)-(A_j-A_s)$ for $i\neq j$. Therefore we obtain
\begin{align*}
|\mb{E}_{x,d\in G}f_1(x+A_1d)\cdots f_s(x+A_sd)|
&\le\left(\mb{E}_{h}\snorm{\partial_{(A_1-A_s)h}f_1}_{U^{s-2}(G)}\right)^{1/2} \\
&\le\left(\mb{E}_{h}\snorm{\partial_{(A_1-A_s)h}f_1}_{U^{s-2}(G)}^{2^{s-2}}\right)^{1/2^{s-1}}\\
&= \snorm{f_1}_{U^{s-1}(G)}.
\end{align*}
The last equality comes from the recursive definition of the Gowers norms as well as the fact that $A_1-A_s$ is an automorphism on $G$. By symmetry, the same holds for $f_2,\ldots,f_s$, completing the proof.
\end{proof}

\section{The \texorpdfstring{$U^3$}{U3}-arithmetic regularity lemma}\label{sec:arith-reg}

From now on until \cref{sec:three-point}, we restrict our attention to the case where $G=\mb{F}_p^n$ and $p$ is an odd prime. The goal of this section is to prove a $U^3$-arithmetic regularity lemma for functions $f\colon G^k\to\mb{C}$. Since $G^k\cong\mb{F}_p^{nk}$, we could apply a standard result (say \cite[{Proposition 3.12}]{Gre06}) to deduce some $U^3$-regularity statement. However such regularity statement would ignore the product structure on $G^k$ which will become very important in our application.

The main novelty of this section is our definition of a $k$-symmetrized quadratic factor which gives an appropriate notion of structured function adapted to the product structure of $G^k$. We then prove \cref{arith-reg-lem}, our $k$-symmetrized $U^3$-arithmetic regularity lemma. The structure of the proof closely follows \cite{Gre06}.

An element $\vec x\in G^k$ is a tuple $\vec x=(x_1,\ldots,x_k)$ with $x_1,\ldots,x_k\in\mb{F}_p^n$. It will simplify the following arguments to introduce the following slightly awkward notation: we view the elements of $G^k$ as $k\times n$ matrices $X$ where the rows of $X$ correspond to the elements of the $k$-tuple. In particular, the element $\vec x\in G^k$ is alternatively represented as \[X=\begin{pmatrix} \frac{\hspace{35pt}}{} x_1^\TT \frac{\hspace{35pt}}{}\\ \vdots \\ \frac{\hspace{35pt}}{} x_k^\TT \frac{\hspace{35pt}}{} \end{pmatrix}. \]

Finally, define $\mc{S}_k$ (respectively, $\mc{S}_k'$) to be the set of symmetric (respectively, skew-symmetric) matrices in $\mb{F}_p^{k\times k}$.

\begin{definition}
A ($k$-)\emph{symmetrized quadratic factor} $\mf{B}=(\mf{B}_1,\mf{B}_2,\mf{B}_3)$ is given by a list $\mf{B}_1=(r_1,\ldots,r_{d_1})$ of column vectors in $\mb{F}_p^n$, a list $\mf{B}_2=(M_1,\ldots,M_{d_2})$ of symmetric matrices in $\mb{F}_p^{n\times n}$, and a list $\mf{B}_3=(N_1,\ldots,N_{d_3})$ of skew-symmetric matrices in $\mb{F}_p^{n\times n}$. The \emph{complexity} of $\mf{B}$ is $(d_1,d_2,d_3)$. We say that $\mf{B}$ has \emph{rank} at least $r$ if $r_1,\ldots,r_{d_1}$ are linearly independent and all nontrivial linear combinations
\[\sum_{i=1}^{d_2}a_iM_i+\sum_{j=1}^{d_3}b_jN_j\]
have $\mb{F}_p$-rank at least $r$. (This is equivalent to the same condition on $M_1,\ldots,M_{d_2}$ and $N_1,\ldots,N_{d_3}$ separately up to an absolute multiplicative constant in the rank.)

A $k$-symmetrized quadratic factor $\mf{B}$ defines maps $\msf{B}_{1,i}\colon G^k\to\mb{F}_p^k$, $\msf{B}_{2,i}\colon G^k\to\mc{S}_k$, and $\msf{B}_{3,i}\colon G^k\to\mc{S}_k'$ given by
\[\msf{B}_{1,i}(X) = Xr_i,\quad\msf{B}_{2,i}(X) = XM_iX^\TT,\quad\msf{B}_{3,i}(X) = XN_iX^\TT.\]
We additionally define
\[\msf{B}_1(X) = (\msf{B}_{1,i}(X))_{i\in[d_1]},\quad\msf{B}_2(X) = (\msf{B}_{2,i}(X))_{i\in[d_2]},\quad\msf{B}_3(X) = (\msf{B}_{3,i}(X))_{i\in[d_3]},\]
\[\msf{B}(X) = (\msf{B}_1(X),\msf{B}_2(X),\msf{B}_3(X)).\]
\end{definition}

For a function $f\colon G^k\to\mb{C}$, we use the notation $\mb{E}[f|\mf{B}]$ to represent the condition expectation of $f$ with respect to $\mf{B}$, or equivalently the projection of $f$ onto $\mf{B}$. Here we abuse notation and use $\mf{B}$ to denote the $\sigma$-algebra generated by the fibers of $\msf{B}$ in $G^k$. Explicitly, $\mb{E}[f|\mf{B}]\colon G^k\to\mb{C}$ is defined by $\mb{E}[f|\mf{B}](X)=\mb{E}_{Y\in\sB^{-1}(\sB(X))}[f(Y)]$.

Finally, we say that a factor $\mf{B}'$ \emph{refines} a factor $\mf{B}$ if the $\sigma$-algebra corresponding to $\mf{B}'$ refines the $\sigma$-algebra corresponding to $\mf{B}$.

The main result of this section is the following arithmetic regularity statement which guarantees that the desired factor is $k$-symmetrized.

\begin{theorem}[Arithmetic regularity lemma]\label{arith-reg-lem}
Fix $k\ge 1$. Let $\delta > 0$ and let $\omega_1,\omega_2\colon \mb{R}^+\to\mb{R}^+$ be arbitrary growth functions (which may depend on $\delta$). Let $G = \mb{F}_p^n$, let $f\colon G^k\to[0,1]$ be a function, and let $(\mf{B}_1^{(0)},\mf{B}_2^{(0)},\mf{B}_3^{(0)})$ be a $k$-symmetrized quadratic factor of complexity $(d_1^{(0)},d_2^{(0)},d_3^{(0)})$. Then there is a refinement $(\mf{B}_1,\mf{B}_2,\mf{B}_3)$ of complexity $(d_1,d_2,d_3)$ and a decomposition $f=f_{\on{str}}+f_{\on{sml}}+f_{\on{psr}}$ such that:
\begin{enumerate}[\quad 1.]
    \item $f_{\on{str}}=\mb{E}[f|\mf{B}]$;
    \item $\|f_{\on{sml}}\|_2\le\delta$;
    \item $\|f_{\on{psr}}\|_{U^3(G^k)}\le1/\omega_2(d_1+d_2+d_3)$;
    \item $f_{\on{str}}$ and $f_{\on{str}}+f_{\on{sml}}$ take values in $[0,1]$ and $f_{\on{psr}}, f_{\on{sml}}$ take values in $[-1,1]$;
    \item the complexity of $\mf{B}$ is $(d_1,d_2,d_3)$ where
    \[d_1,d_2,d_3\le C(k,\delta,\omega_1,\omega_2,d^{(0)}_1,d^{(0)}_2,d^{(0)}_3)\] for a fixed function $C$;
    \item the rank of $\mf{B}$ is at least $\omega_1(d_1+d_2+d_3)$.
\end{enumerate}
\end{theorem}

The proof closely follows the proof of arithmetic regularity given in \cite{Gre05}; the only additional ingredient is guaranteeing at each stage that the factor introduced is $k$-symmetrized.

\begin{lemma}\label{thm:energy-increment}
Fix $k\ge 1$. Let $\delta>0$. There exists $\epsilon>0$ such that the following holds. Let $\mf{B}^{(0)}=(\mf{B}^{(0)}_1,\mf{B}^{(0)}_2,\mf{B}^{(0)}_3)$ be a $k$-symmetrized quadratic factor with complexity $(d_1,d_2,d_3)$ and let $f\colon G^k\to[-1,1]$ be a function such that \[\|f-\mb{E}[f|\mf{B}^{(0)}]\|_{U^3(G^k)}\ge\delta.\] Then there exists a refinement $\mf{B}=(\mf{B}_1,\mf{B}_2,\mf{B}_3)$ with complexity at most $(d_1+k,d_2+\binom{k+1}2,d_3+\binom{k}2)$ such that \[\|\mb{E}[f|\mf{B}]\|_2^2\ge\|\mb{E}[f|\mf{B}^{(0)}]\|_2^2+\epsilon^2.\]
\end{lemma}

\begin{proof}
By the inverse theorem for the Gowers $U^3$-norm applied to $G^k\cong\mb{F}_p^{nk}$ (see \cite{Gre06,GT08}), there exist $\epsilon>0$ (only depending on $\delta$), a vector $r\in\mb{F}_p^{nk}$, and a symmetric matrix $M\in\mb{F}_p^{nk\times nk}$ such that \[\left|\mb{E}_{x\in \mb{F}_p^{nk}}\left(f(x)-\mb{E}[f|\mf{B}^{(0)}](x)\right)e_p(r^\TT x+x^\TT M x)\right|\ge\epsilon.\]

Say $r=(r_1,\ldots, r_k)\in\mb{F}_p^{nk}$ where $r_1,\ldots,r_k\in\mb{F}_p^n$ and $M=(M_{ij})_{i,j\in[k]}\in\mb{F}_p^{nk\times nk}$ where $M_{ij}\in \mb{F}_p^{n\times n}$. Note that the matrices $M_{ii}$ are symmetric, while $M_{ij}=M_{ji}^\TT$. For $i<j$, write $M_{ij}=M_{ij}'+M_{ij}''$ where $M_{ij}'$ is symmetric and $M_{ij}''$ is skew-symmetric. (Here we use that $p>2$.)

We define the factor $\mf{B}$ by appending the vectors $r_1,\ldots, r_k$ to the list $\mf{B}^{(0)}_1$, appending the symmetric matrices $(M_{ii})_{i\in [k]}$ and $(M_{ij}')_{i<j}$ to the list $\mf{B}^{(0)}_2$, and appending the skew-symmetric matrices $M_{ij})_{i<j}$ to the list $\mf{B}^{(0)}_3$. To conclude, all that remains to show is that \[\|\mb{E}[f|\mf{B}]\|_2^2\ge\|\mb{E}[f|\mf{B}^{(0)}]\|_2^2+\epsilon^2.\]

Define $g\colon \left(\mb{F}_p^n\right)^k\to\mb{C}$ by $g(x)=e_p(r^\TT x+x^\TT M x)$. Note that $\|g\|_2=1$ and crucially that $g$ is $\mf{B}$-measurable by the simple equality \[g(x_1,\ldots,x_k)=e_p\left(\sum_{i=1}^k r_i^\TT x_i+\sum_{i=1}^k x_i^\TT M_{ii}x_i+2\sum_{i<j}x_i^\TT M_{ij}'x_j+2\sum_{i<j}x_i^\TT M_{ij}''x_j\right).\]

Now the desired inequality follows from the Pythagorean theorem and Cauchy--Schwarz inequality since
\begin{align*}
\|\mb{E}[f|\mf{B}]\|_2^2 - \|\mb{E}[f|\mf{B}^{(0)}]\|_2^2
&=\|\mb{E}[f|\mf{B}]-\mb{E}[f|\mf{B}^{(0)}]\|^2\\
&\ge \left|\left\langle\mb{E}[f|\mf{B}]-\mb{E}[f|\mf{B}^{(0)}], g\right\rangle\right|^2\\
&=\left|\left\langle f-\mb{E}[f|\mf{B}^{(0)}], \mb{E}[g|\mf{B}]\right\rangle\right|^2\\
&\ge\epsilon^2.\qedhere
\end{align*}
\end{proof}

\begin{lemma}\label{thm:making-factors-high-rank}
Fix $k\ge 1$. Let $\omega\colon\mb{R}^+\to\mb{R}^+$ be an arbitrary growth function. There exists a growth function $\tau\colon\mb{R}^+\to\mb{R}^+$ such that the following holds. Let $\mf{B}=(\mf{B}_1,\mf{B}_2,\mf{B}_3)$ be a $k$-symmetrized quadratic factor with complexity $(d_1,d_2,d_3)$. There exists a refinement $\mf{B}'=(\mf{B}_1,\mf{B}_2,\mf{B}_3)$ with complexity $(d'_1,d'_2,d'_3)$ that satisfies the following:
\begin{enumerate}[\quad 1.]
\item the rank of $\mf{B}'$ is at least $\omega(d_1'+d_2'+d_3')$;
\item $d_2'\le d_2$ and $d_3'\le d_3$ and $d_1'\le\tau(d_1+d_2+d_3)$.
\end{enumerate}
\end{lemma}

\begin{proof}
Consider a $k$-symmetrized quadratic factor $\mf{B}=(\mf{B}_1,\mf{B}_2,\mf{B}_3)$ defined by $\mf{B}_1=(r_1,\ldots,r_{d_1})$ and $\mf{B}_2=(M_1,\ldots,M_{d_2})$ and $\mf{B}_3=(N_1,\ldots,N_{d_3})$. (Recall that the $r_i$ are vectors of length $n$ while the $M_i$ are symmetric $n\times n$ matrices and the $N_i$ are skew-symmetric $n\times n$ matrices.)

If the rank of $\mf{B}$ is less than $r$, then either the $r_1,\ldots, r_{d_1}$ are linearly dependent or there exists a non-trivial linear combination \[\sum_{i=1}^{d_2}a_iM_i+\sum_{j=1}^{d_3}b_jN_j\]
that has $\mb{F}_p$-rank less than $r$.

We do the following. First if there is some linear combination with rank less than $r$, then choose vectors $s_1,\ldots, s_{r-1},t_1,\ldots,t_{r-1}$ such that \[\sum_{i=1}^{r-1} s_i t_i^\TT = \sum_{i=1}^{d_2}a_iM_i+\sum_{j=1}^{d_3}b_jN_j.\] Add $s_1,\ldots, s_{r-1},t_1,\ldots,t_{r-1}$ to $\mf{B}_1$ and remove the first $M_i$ or $N_j$ with nonzero coefficient (i.e., if $a_1=\cdots=a_{i-1}=0$ while $a_i\neq 0$, then remove $M_i$ from $\mf{B}_2$; if $a_1=\cdots=a_{d_2}=b_1=\cdots=b_{j-1}=0$ and $b_j\neq 0$, then remove $N_j$ from $\mf{B}_3$). Then remove any element of the modified $\mf{B}_1$ that is linearly dependent on the previous vectors in $\mf{B}_1$. Note that the factor produced refines the original factor.

We iterate the above process, producing a sequence of $k$-symmetrized quadratic factors $\mf{B}=\mf{B}^{(0)},\mf{B}^{(1)},\ldots,\mf{B}^{(M)}$ as follows. Suppose that $\mf{B}^{(m)}$ has complexity $(d_1^{(m)},d_2^{(m)},d_3^{(m)})$. If $\mf{B}^{(m)}$ has rank at least $\omega(d_1^{(m)}+d_2^{(m)}+d_3^{(m)})$, then halt and set $m=M$. Otherwise refine $\mf{B}^{(m)}$ to $\mf{B}^{(m+1)}$ as described above. Note that $M\le d_2+d_3+1$ since every step (except possibly the first) reduces $d_2^{(m)}+d_3^{(m)}$ by 1. Furthermore, one can easily see that $d_1^{(M)}$ is bounded by some function of $d_1,d_2,d_3$ and $\omega$, as desired.
\end{proof}

\begin{lemma}\label{thm:weak-regularity}
Fix $k\ge 1$. Let $\delta>0$ and let $\omega\colon\mb{R}^+\to\mb{R}^+$ be an arbitrary growth function. Let $\mf{B}^{(0)}=(\mf{B}^{(0)}_1,\mf{B}^{(0)}_2,\mf{B}^{(0)}_3)$ be a $k$-symmetrized quadratic factor with complexity $(d^{(0)}_1,d^{(0)}_2,d^{(0)}_3)$ and let $f\colon G^k\to[0,1]$ be a function. Then there exists a refinement $\mf{B}=(\mf{B}_1,\mf{B}_2,\mf{B}_3)$ and a decomposition $f=f_{\on{str}}+f_{\on{psr}}$ such that:
\begin{enumerate}[\quad 1.]
    \item $f_{\on{str}}=\mb{E}[f|\mf{B}]$;
    \item $\|f_{\on{psr}}\|_{U^3(G^k)}\le\delta$;
    \item $f_{\on{str}}$ takes values in $[0,1]$ and $f_{\on{psr}}$ takes values in $[-1,1]$;
    \item the complexity of $\mf{B}$ is $(d_1,d_2,d_3)$ where
    \[d_1,d_2,d_3\le C(k,\delta,d^{(0)}_1,d^{(0)}_2,d^{(0)}_3)\] for a fixed function $C$;
    \item the rank of $\fB$ is at least $\omega(d_1+d_2+d_3)$.
\end{enumerate}
\end{lemma}

\begin{proof}
This follows immediately by iterating \cref{thm:energy-increment} and \cref{thm:making-factors-high-rank} at most $\epsilon(\delta)^{-2}$ times.

In particular, we construct a sequence of $k$-symmetrized quadratic factors $\mf{B}^{(0)},\mf{B}^{(1)},\ldots,\mf{B}^{(M)}$ each refining the last as follows. If $\|f-\mb{E}[f|\mf{B}^{(m)}]\|_{U^3(G^k)}<\delta$, halt the process and set $M=m$. Otherwise, let $\tilde{\mf{B}}^{(m+1)}$ be the factor produced by applying \cref{thm:energy-increment} to $\mf{B}^{(m)}$ and $f$ with parameter $\delta$. Then let $\fB^{(m+1)}$ be the factor produced by applying \cref{thm:making-factors-high-rank} to $\tilde{\fB}^{(m+1)}$ with parameter $\omega$. By definition, at every step of this process, the rank of $\fB^{(m)}$ is at least $\omega_1(d_1^{(m)}+d_2^{(m)}+d_3^{(m)})$ where $(d_1^{(m)},d_2^{(m)},d_3^{(m)})$ is the complexity of $\fB$.

Since \[\|\mb{E}[f|\mf{B}^{(m+1)}\|_2^2\geq\|\mb{E}[f|\tilde{\mf{B}}^{(m+1)}\|_2^2\geq\|\mb{E}[f|\mf{B}^{(m)}\|_2^2+\epsilon(\delta)^2\] and this quantity is bounded between 0 and 1, we conclude that the process must stop after $M\le\epsilon(\delta)^{-2}$ steps.

At the conclusion of this process, we have produced a $k$-symmetrized quadratic factor $\mf{B}^{(M)}$ that refines $\mf{B}^{(0)}$ such that $\|f-\mb{E}[f|\mf{B}^{(M)}]\|_{U^3(G^k)}<\delta$. Defining $f_{\on{str}}=\mb{E}[f|\mf{B}^{(M)}]$ and $f_{\on{psr}}=f-\mb{E}[f|\mf{B}^{(M)}]$ gives the desired result.
\end{proof}

\begin{proof}[Proof of \cref{arith-reg-lem}]
The desired result follows by iterating \cref{thm:weak-regularity} at most $\delta^{-2}$ times.

In particular, we construct a sequence of $k$-symmetrized quadratic factors $\mf{B}^{(0)},\mf{B}^{(1)},\ldots,\mf{B}^{(M)}$ each refining the last as follows. If $\|\E[f|\fB^{(m)}]-\mb{E}[f|\mf{B}^{(m-1)}]\|_2^2<\delta^2$, halt the process and set $M=m$. Otherwise, let $\mf{B}^{(m+1)}$ be the factor produced by applying \cref{thm:weak-regularity} to $\mf{B}^{(m)}$ and $f$ with parameter $1/\omega_2(d_1^{(m)}+d_2^{(m)}+d_3^{(m)})$ and growth function $\omega_1$.

Note that by the Pythagorean theorem, \[\|\E[f|\fB^{(m)}]-\mb{E}[f|\mf{B}^{(m-1)}]\|_2^2=\|\E[f|\fB^{(m)}]\|_2^2-\|\E[f|\fB^{(m-1)}]\|_2^2.\] Since these $L^2$-norms are bounded between 0 and 1, we see that the process must stop after $M\le\delta^{-2}$ steps.

At the conclusion of this process, we have produced a $k$-symmetrized quadratic factor $\fB^{(M-1)}$ that refines $\fB^{(0)}$ with complexity $(d_1^{(M-1)},d_2^{(M-1)},d_3^{(M-1)})$ and rank at least $\omega_1(d_1^{(M-1)}+d_2^{(M-1)}+d_3^{(M-1)})$. Defining $f_{\on{str}}=\E[f|\fB^{(M-1)}]$ and $f_{\on{sml}}=\E[f|\fB^{(M)}]-\E[f|\fB^{(M-1)}]$ and $f_{\on{psr}}=f-\E[f|\fB^{(M)}]$ gives the desired result.
\end{proof}

\section{Equidistribution and counting lemma}\label{sec:equidistribution}

The goal of this section is to study the counts of matrix patterns of the form $\{0,M_1,M_2,M_1+M_2\}$ in the ``structured term'' $f_{\on{str}}$. Recall that a $k$-symmetrized quadratic factor $\fB$ defines a map $\sB\colon G^k\to \left(\F_p^k\right)^{d_1}\times \mc{S}_k^{d_2}\times \mc{S}_k'^{d_3}$ where $\mc{S}_k$ and $\mc{S}_k'$ are the spaces of $k\times k$ symmetric (resp. skew-symmetric) matrices. We call the fibers of this map \emph{atoms} of $\fB$.

Understanding the counts of patterns in $f_{\on{str}}$ is equivalent to understanding how occurrences of these patterns are distributed among tuples of atoms. The first result of this section is simply that the atoms of $\fB$ are approximately the same size; in other words, as $X\in G^k$ varies, $\sB(X)$ is equidistributed in $\left(\F_p^k\right)^{d_1}\times \mc{S}_k^{d_2}\times \mc{S}_k'^{d_3}$.

The main result of this section describes how the 4-tuple $(\sB(X),\sB(X+M_1D),\sB(X+M_2D),\sB(X+(M_1+M_2)D))$ is distributed as $X,D\in G^k$ vary. This 4-tuple is not equidistributed across all possible 4-tuples of atoms, instead it is equidistributed on a certain linear subspace. We need a somewhat unfortunate amount of notation in this section to describe this linear subspace.

Note that this is also the place where the ``mysterious'' spectral condition that $M_1M_2^{-1}$ has no pair of eigenvalues that are negatives of each other appears. It turns out that the dimension of the space that the relevant 4-tuples are equidistributed over changes depending on whether or not this spectral condition is satisfied.

Finally in this section we restrict our attention to matrix patterns of the form $\{0,I,J,I+J\}$ where $I = I_{k\times k}$ is the identity, $J,I-J,I+J$ are invertible, and that $J$ satisfies the spectral condition (that no pair of eigenvalues of $J$ over $\ol{\mb{F}}_p$ are negatives of each other). By a change of variables, all cases can be reduced to this one.

\subsection{Equidistribution results}\label{sub:equidistribution}
We first quote the following result on the equidistribution in $\mb{F}_p^n$.
\begin{proposition}[{\cite[Lemma~4.2]{Gre06}}]\label{prop:equidistribution}
Define $\Gamma(x)=(r_1^\TT x,\ldots,r_{d_1}^\TT x)$ and $\Phi(x)=(x^\TT M_1x,\ldots,x^\TT M_{d_2}x)$ where the $M_i$ are symmetric. Furthermore suppose that $\{r_i\}_{i\in[d_1]}$ are linearly independent and for any nonzero vector $(\lambda_1,\ldots,\lambda_{d_2})$ in $\mb{F}_p^{d_2}$ we have $\on{rank}(\sum_{i=1}^{d_2}\lambda_iM_i)\ge r$. Then for any $a\in\mb{F}_p^{d_1}$ and $b\in\mb{F}_p^{d_2}$ we have
\[\mb{P}_{x\in\mb{F}_p^n}[\Gamma(x) = a, \Phi(x) = b] = p^{-d_1-d_2}+O(p^{-r/2}).\]
\end{proposition}

Note \cite{Gre06} only states the above for $\mb{F}_5^n$ but the proof in general is completely analogous. Given this we can immediately derive the necessary equidistribution result on factors for the specialized factors constructed in the previous section. 
\begin{proposition}\label{prop:factor-equidistribution}
Let $\mf{B}$ be a $k$-symmetrized quadratic factor with rank at least $r$. Then
\[\mb{P}_{X\in\mb{F}_p^{k\times n}}[\msf{B}(X) = ((v_i)_{i\in[d_1]},(U_i)_{i\in[d_2]},(V_i)_{i\in[d_3]})] = p^{-kd_1-\binom{k+1}{2}d_2-\binom{k}{2}d_3}+O(p^{-r/2})\]
for all $v_i\in\mb{F}_p^k$, $U_i\in \mc{S}_k$, and $V_i\in \mc{S}_k'$.
\end{proposition}
\begin{proof}
This is immediate if one treats $X\in\mb{F}_p^{k\times n}$ as a $kn$-dimensional vector. In particular for each $U_{i}$ consider the family $M_i$ of $\binom{k+1}{2}$ block matrices where all $n$ by $n$ blocks are zero except for either a diagonal block labeled $U_i$ or a pair of block symmetric with respect to the diagonal such that blocks are labeled $U_i$. Similarly for each $V_{i}$ consider the family $N_i$ of $\binom{k}{2}$ block matrices where all $n$ by $n$ blocks are zero for a pair of block symmetric with respect to the diagonal such that the block above the diagonal is labeled $V_i$ and below the diagonal is labeled $-V_i$. Note that the resulting quadratic forms are easily seen to be high rank using that the $U_i$ and $V_i$ initially where high rank. Now the desired equidistribution statement is equivalent to equidistribution of $X^\TT WX$ for all $W\in M_i,N_i$ as well the linear forms specified by $v_i$. This now follows immediately from \cref{prop:equidistribution}.
\end{proof}

Say that a random variable is \emph{$\epsilon$-equidistributed} if it takes each value in its range with equal probability within a \emph{multiplicative} error of $\epsilon$. A convenient property of this definition is that it is preserved under linear maps.

\begin{lemma}\label{lem:projection-equidistribution}
Suppose $\mbf{x}$ is a random variable taking values in $\mb{F}_p^r$ satisfying
\[\sup_{a\in\mb{F}_p^r}|p^r\mb{P}[\mbf{x} = a]-1|\le\epsilon\]
and $L\colon \mb{F}_p^r\to\mb{F}_p^s$ is a linear map with image of dimension $t$. Then for any $a\in L\mb{F}_p^n$ we have $|p^t\mb{P}[L\mbf{x} = a]-1|\le\epsilon$, whereas if $a\notin L\mb{F}_p^n$ then $\mb{P}[L\mbf{x} = a] = 0$.
\end{lemma}
\begin{proof}
This follows immediately from the fact that the preimage of every point in $L\mb{F}_p^r$ has size $p^{r-t}$.
\end{proof}

We now explicitly define the relevant lattice that the image of our pattern under $\msf{B}$ will equidistribute over, in order to state the main result of this section. Recall we have a given $J\in\mb{F}_p^{k\times k}$. Let
\[\Xi_J = \{A\in\mb{F}_p^{k\times k}: (JA)^\TT = JA\},\]
and let
\begin{align*}
\Lambda_J &= \{(-A,-A(I+J)(I-J)^{-1},A(I+J)(I-J)^{-1},A)\colon A^\TT = +A, A\in\Xi_J\},\\
\Lambda_J' &= \{(-A,-A(I+J)(I-J)^{-1},A(I+J)(I-J)^{-1},A)\colon A^\TT = -A, A\in\Xi_J\}.
\end{align*}

Also, let
\[\Psi_J = \{(x_1,x_2,x_3,x_4)\in(\mb{F}_p^k)^4\colon x_1-x_2-x_3+x_4 = 0, x_4-x_2=J(x_2-x_1)\}.\]
We make $\F_p^{k\times k}$ an inner product space with the standard inner product
\[\sang{A,B} = \sang{A,B}_{\on{HS}} = \on{tr}(A^\TT B)\]
on $\mb{F}_p^{k\times k}$. We extend this inner product to $\left(\F_p^{k\times k}\right)^4$ in the natural way, that is,\[\sang{(X_1,X_2,X_3,X_3),(Y_1,Y_2,Y_3,Y_4)}=\sang{X_1,Y_1}+\sang{X_2,Y_2}+\sang{X_3,Y_3}+\sang{X_4,Y_4}.\]

We wish to study the equidistribution of the tuple
\[(\msf{B}(X),\msf{B}(X+D),\msf{B}(X+JD),\msf{B}(X+(I+J)D))\]
as $X,D$ range over $\mb{F}_p^{k\times n}$, for a $k$-symmetrized quadratic factor $\mf{B}$. Ultimately, we will find that the components corresponding to each $\msf{B}_{1,i},\msf{B}_{2,i},\msf{B}_{3,i}$ are all ``independent'', and that each equidistributes in the following way:
\begin{align*}
(\msf{B}_{1,i}(X),\msf{B}_{1,i}(X+D),\msf{B}_{1,i}(X+JD),\msf{B}_{1,i}(X+(I+J)D))&\qquad\text{equidistributes on}\qquad\Psi_J,\\
(\msf{B}_{2,i}(X),\msf{B}_{2,i}(X+D),\msf{B}_{2,i}(X+JD),\msf{B}_{2,i}(X+(I+J)D))&\qquad\text{equidistributes on}\qquad\Lambda_J^\perp\cap(\mc{S}_k)^4, \\
(\msf{B}_{3,i}(X),\msf{B}_{3,i}(X+D),\msf{B}_{3,i}(X+JD),\msf{B}_{3,i}(X+(I+J)D))&\qquad\text{equidistributes on}\qquad\Lambda_J'^\perp\cap(\mc{S}_k')^4.
\end{align*}
Here the $\perp$ means the orthogonal subspace with respect to the inner product defined above. For ease of notation, we will write $\Lambda_J^\perp$ for $\Lambda_J^\perp\cap(\mc{S}_k)^4$ and $\Lambda_J'^\perp$ for $\Lambda_J'^\perp\cap(\mc{S}_k')^4$.

\begin{theorem}\label{thm:counting}
Suppose $J\in\mb{F}_p^{k\times k}$ is such that $J,I-J,I+J$ are invertible and $J$ has no pair of eigenvalues that are negatives of each other (over $\ol{\mb{F}}_p$), let $G = \mb{F}_p^n$, and suppose $\mf{B}$ is a $k$-symmetrized quadratic factor of rank $r$. Then for any $a\in \Psi_J^{d_1}\times (\Lambda_J^\perp)^{d_2} \times (\Lambda_J'^\perp)^4)^{d_3}$

\begin{align*}
&\mb{P}_{X,D\in G^k}[(\msf{B}(X),\msf{B}(X+D),\msf{B}(X+JD),\msf{B}(X+(I+J)D))=a]\\
&= p^{-d_1\dim(\Psi_J)-d_2\dim(\Lambda_J^\perp) - d_3\dim(\Lambda_J'^\perp)  }(1+O(p^{-r/2+2kd_1+(2\binom{k+1}{2}+k^2)d_2+(2\binom{k}{2}+k^2)d_3}))
\end{align*}
\end{theorem}

The approach is similar to the proof of \cref{prop:factor-equidistribution}. We want to consider $(X,D)$ as a $2kn$-dimensional vector and apply \cref{prop:equidistribution}, but now some linear dependencies will appear.\footnote{These linear dependencies appear for the same reason that $(x^2,(x+d)^2,(x+2d)^2,(x+3d)^2)$ satisfies a linear equation.} We will instead apply equidistribution on a set of ``abstractly independent'' forms to which we can indeed apply \cref{prop:equidistribution}. Then we realize $(\msf{B}(X),\msf{B}(X+D),\msf{B}(X+JD),\msf{B}(X+(I+J)D))$
as the image of those elements under a linear map, and apply \cref{lem:projection-equidistribution}.

As a first step, we state the necessary equidistribution over these ``abstract atoms''. For convenience, given a $k$-symmetrized quadratic factor $\fB$, define an attached map
\[\msf{B}'(X,D) = ((XM_iD^\TT)_{i\in[d_2]},(XN_iD^\TT)_{i\in[d_3]}).\]
\begin{proposition}\label{prop:abstract-atom-equidistribution}
Suppose $\mf{B}$ is a $k$-symmetrized quadratic factor of rank $r$. Then for any $a\in((\mb{F}_p^k)^{d_1}\times(\mc{S}_k)^{d_2}\times(\mc{S}_k')^{d_3})^2\times(\mb{F}_p^{k\times k})^{d_2+d_3}$ we have
\[\mb{P}_{X,D\in G^k}[(\msf{B}(X),\msf{B}(D),\msf{B}'(X,D))=a] = p^{-2kd_1-(2\binom{k+1}{2}+k^2)d_2-(2\binom{k}{2}+k^2)d_3}+O(p^{-r/2}).\]
\end{proposition}
This follows immediately by applying \cref{prop:equidistribution} to $(X,D)$ viewed as an element of $\F_p^{2kn}$. The proof is exactly analogous to the proof of \cref{prop:factor-equidistribution} from \cref{prop:equidistribution}.

To complete the proof of \cref{thm:counting}, we note that the desired map, $(X,D)\mapsto (\msf{B}(X),\msf{B}(X+D),\msf{B}(X+JD),\msf{B}(X+(I+J)D)$, can be written as the map $(X,D)\mapsto (\msf{B}(X),\msf{B}(D),\msf{B}'(X,D))$ composed with a linear transformation. For example,
\[(X+JD)M_i(X+JD)^\TT = XM_iX^\TT + J(DM_iX^\TT) + XM_iD^\TT J^\TT + J(DM_iD^\TT)J^\TT\]
and $XM_iD^\TT = (DM_iX^\TT)^\TT$.
Therefore it suffices to understand the linear constraints induced by this last linear transformation.

\subsection{Deriving the linear constraints}\label{sub:algebraic-inputs}
To this end we prove the following abstract linear algebra statement, which essentially encodes the eigenvalue condition in \cref{thm:main}.
\begin{lemma}\label{lem:A-algebra}
If $A\in\mb{F}_p^{k\times k}$ is invertible and has no pair of eigenvalues (over $\ol{\mb{F}}_p$) which are negatives of each other, then $A\in\mb{F}_p[A^2]$.
\end{lemma}
\begin{proof}
Given any matrix $M$, let $Q_M\in\mb{F}_p[t]$ be the monic polynomial of minimum degree satisfying $Q_M(M) = 0$ (this exists by the Cayley--Hamilton theorem and the fact that $\mb{F}_p[t]$ is a principal ideal domain). Then $\mb{F}_p[M]\cong\mb{F}_p[t]/(Q_M(t))$ as $\mb{F}_p$-algebras, and the dimension as an $\mb{F}_p$-vector space is $\deg Q_M$.

We clearly have $A\in\mb{F}_p[A^2]$ if and only if $\mb{F}_p[A] = \mb{F}_p[A^2]$, which will certainly follow from
\[\dim_{\mb{F}_p}\mb{F}_p[A]\le\dim_{\mb{F}_p}\mb{F}_p[A^2]\]
due to the obvious containment. Now note that if $g(A^2) = 0$ then $Q_A(t)|g(t^2)$ in $\mb{F}_p[t]$. By hypothesis, we have $\gcd(Q_A(t),Q_A(-t)) = 1$, hence $Q_A(t)Q_A(-t)|g(t^2)$. Thus
\[2\dim_{\mb{F}_p}\mb{F}_p[A] = \deg (Q_A(t)Q_A(-t))\le 2\deg g.\]
Since this holds for all such $g$, it in particular holds for $g = Q_{A^2}$, which implies the result.
\end{proof}

Next we need the following abstract matrix equation which is used to derive the desired equidistribution statement.
\begin{lemma}\label{lem:f-eq}
Suppose $J\in\mb{F}_p^{k\times k}$ is such that $J,I-J,I+J$ are invertible and $J$ has no pair of eigenvalues that are negatives of each other.
\begin{itemize}
    \item Let $M = M^\TT$ be nonzero. Then $(A_1,A_2,A_3,A_4)\in(\mc{S}_k)^4$ and 
    \begin{align*}
      \on{tr}&(A_1^\TT XMX^\TT  + A_2^\TT (X+D)M (X+D)^\TT  + A_3^\TT (X+JD) M (X+JD)^\TT \\
      &\qquad\qquad+ A_4^\TT (X+(I+J)D)M (X+(I+J)D)^\TT) =0  
    \end{align*}
    for all $X,D\in (\mb{F}_p^{n})^k$ if and only if $(A_1,A_2,A_3,A_4)\in \Lambda_J$.
    \item Let $M = -M^\TT$ be nonzero. Then $(A_1,A_2,A_3,A_4)\in(\mc{S}_k')^4$ and 
    \begin{align*}
      \on{tr}&(A_1^\TT XMX^\TT  + A_2^\TT (X+D)M (X+D)^\TT  + A_3^\TT (X+JD) M (X+JD)^\TT \\
      &\qquad\qquad+ A_4^\TT (X+(I+J)D)M (X+(I+J)D)^\TT) =0  
    \end{align*}
    for all $X,D\in (\mb{F}_p^{n})^k$ if and only if $(A_1,A_2,A_3,A_4)\in \Lambda_J'$.
\end{itemize}
\end{lemma}
\begin{proof}
We prove the first claim as the second is analogous. Note that 
\begin{align*}
      \on{tr}&(A_1^\TT XMX^\TT  + A_2^\TT (X+D)M (X+D)^\TT  + A_3^\TT (X+JD) M (X+JD)^\TT \\
      &\qquad\qquad+ A_4^\TT (X+(I+J)D)M (X+(I+J)D)^\TT) =0  
    \end{align*}
implies that 
\[\on{tr}((A_1^\TT +A_2^\TT +A_3^\TT +A_4^\TT)XMX^\TT) = 0,\]
\[\on{tr}(A_2^\TT (D)M (D)^\TT  + A_3^\TT (JD)M (JD)^\TT 
     + A_4^\TT ((I+J)D)M ((I+J)D)^\TT) = 0,\]
taking $D = 0$ and $X = 0$ respectively. Using that $\on{tr}(\cdot)$ is additive along with the initial condition, we derive
\begin{align*}
\on{tr}&((A_1^\TT +A_2^\TT +A_3^\TT +A_4^\TT)XMX^\TT) = 0, \\
\on{tr}&(A_2^\TT (D)M (D)^\TT  + A_3^\TT (JD)M (JD)^\TT 
     + A_4^\TT ((I+J)D)M ((I+J)D)^\TT) =0,\\
\on{tr}&(A_2^\TT (DMX^\TT + XMD^\TT) + A_3^\TT (JDMX^\TT + XM(JD)^\TT) \\
     &\qquad \qquad+ A_4^\TT ((J+I)DMX^\TT + XM((J+I)D)^\TT)=0.
\end{align*}
Using that trace $\on{tr}(A^\TT B^\TT) = \on{tr}(BA) = \on{tr}(AB)$ the above conditions are equivalent to 
\begin{align*}
\on{tr}&((A_1^\TT +A_2^\TT +A_3^\TT +A_4^\TT)XMX^\TT) = 0 \\
\on{tr}&((A_2^\TT + J^\TT A_3^\TT J + (I+J)^\TT A_4^\TT (I+J))(DMD)^\TT) =0\\
\on{tr}&((2A_2^\TT + 2A_3^\TT J + 2A_4^\TT (I+J)) (DMX^\TT))=0.
\end{align*}
To derive the last, we used both $A_i^\TT = A_i$ and $M^\TT = M$. (In fact, one obtains identical equations in the skew-symmetric case.)

Since $M$ is nonzero and symmetric we have that $\{XMX^\TT\}$, $\{DMD^\TT\}$ each span the space of all $k\times k$ symmetric matrices while $\{DMX^\TT\}$ spans the space of all $k\times k$ matrices. This implies that
\begin{align*}
A_1^\TT +A_2^\TT +A_3^\TT +A_4^\TT &= 0 \\
A_2^\TT + J^\TT A_3^\TT J + (I+J)^\TT A_4^\TT (I+J) &= 0\\
A_2^\TT + A_3^\TT J + A_4^\TT (I+J) &= 0
\end{align*}
since the $A_i$ are symmetric.

Note that the second and third equations imply that 
\[(J^\TT-I) A_2^\TT = A_4^\TT (I+J).\]
Noting that $A_2^\TT, A_4^\TT$ are symmetric we find that
\[A_2^\TT = (J^\TT-I)^{-1}A_4^\TT (I+J) = (I+J^\TT) A_4^\TT (J-I)^{-1}.\]
This is equivalent to 
\[(J^\TT)^2A_4^\TT = A_4^\TT J^2.\]
It follows that for every polynomial $Q(t)\in\mb{F}_p[t]$ we have
\[Q((J^\TT)^2)A_4^\TT = A_4^\TT Q(J^2).\]
Now $J\in \mb{F}_p[J^2]$ by \cref{lem:A-algebra}, so it follows that 
\[J^\TT A_4^\TT = A_4^\TT J.\]

A similar combination of the second and third equations implies that
\[J^\TT A_2^\TT = -A_3^\TT J\]
and therefore 
\[(J^\TT)^{-1}A_3^\TT J = (J^\TT)A_3^\TT J^{-1}.\]
This similarly implies that $J^\TT A_3^\TT = A_3^\TT J
$. Therefore we deduce
\begin{align*}
A_1^\TT +A_2^\TT +A_3^\TT +A_4^\TT &= 0 \\
A_2^\TT + A_3^\TT J^2 + A_4^\TT (I+J)^2 &=0\\
A_2^\TT + A_3^\TT J + A_4^\TT (I+J) &=0.
\end{align*}
Subtracting the last two equations gives that 
\[A_3^\TT = A_4^\TT(I+J)(I-J)^{-1}.\]
Substituting into the last equation gives 
\[A_2^\TT = A_4^\TT(I+J)(-I-J(I-J)^{-1}) = -A_3^\TT.\]
Finally using the first equation this implies that $A_1^\TT = -A_4^\TT$. Therefore we have proven that $(A_1,A_2,A_3,A_4)\in \Lambda_J$. The reverse implication is a straightforward calculation which we omit.
\end{proof}
We are now ready to prove \cref{thm:counting}.
\begin{proof}[Proof of \cref{thm:counting}]
This is almost immediate from \cref{prop:abstract-atom-equidistribution,lem:projection-equidistribution,lem:f-eq}. The noted fact that the desired function is the image of the function in \cref{prop:abstract-atom-equidistribution} under a linear mapping along with \cref{lem:projection-equidistribution} demonstrates that there is equidistribution over some subspace (with the multiplicative error term preserved).

This subspace is precisely the span of all possible vectors $(\msf{B}(X),\msf{B}(X+D),\msf{B}(X+JD),\msf{B}(X+(I+J)D))$. Due to the independence demonstrated in \cref{prop:abstract-atom-equidistribution}, we see that the subspace factors as a direct sum.

\cref{lem:f-eq} characterizes the resulting vector spaces for $\msf{B}_{2,i},\msf{B}_{3,i}$, since it demonstrates the form of all orthogonal vectors in corresponding host spaces (either tuples of symmetric or skew-symmetric matrices). For $\msf{B}_1$, it is easy to check that vectors of the form $(Xr,(X+D)r,(X+JD)r,(X+(I+J)D)r)$ span the space $\Psi_J$ when $r\in\mb{F}_p^n\setminus 0$. Indeed, the orthogonal vectors of the form $(\vec{a}_1,\vec{a}_2,\vec{a}_3,\vec{a}_3)$ are precisely those that satisfy
\[\vec{a}_1+\vec{a}_2+\vec{a}_3+\vec{a}_4 = \vec{a}_2+\vec{a}_3J+\vec{a}_4(I+J) = 0.\]
All such vectors are spans of $(\vec{t},-\vec{t},-\vec{t},\vec{t})$ and $(J\vec{t},-(I+J)\vec{t},0,\vec{t})$, which matches the orthogonal space of $\Psi_J$. This completes the proof.
\end{proof}

\section{Popular differences for \texorpdfstring{$4$}{4}-point patterns}\label{sec:4-point-popular-differences}

We now have the tools to prove the following popular difference result for four-point patterns. This theorem immediately implies the result stated in the introduction, \cref{thm:main}.

\begin{theorem}\label{thm:main-functional}
Fix $k\ge 1$ and an odd prime $p$. Let $M_1,M_2$ be $k\times k$ matrices with coefficients in $\mb{F}_p$ such that $M_1$, $M_2$, $M_1-M_2$, and $M_1+M_2$ are invertible and no pair of eigenvalues of $M_1M_2^{-1}$ (viewed over $\overline{\mb{F}}_p$) are negatives of each other. For any $\alpha,\epsilon > 0$, letting $G = \mb{F}_p^n$ for $n > n_0(\alpha,\epsilon,p)$, if $f\colon G^k\to[0,1]$ satisfies $\mb{E}_{X\in G^k}f(X)\ge\alpha$ then there are $\Omega_{\alpha,\epsilon,p}(p^{kn})$ values $D\in G^k$ such that
\[\mb{E}_{X\in G^k}f(X)f(X+M_1D)f(X+M_2D)f(X+(M_1+M_2)D)\ge\alpha^4-\epsilon.\]
\end{theorem}

Given the previous developments the proof is similar to that of \cite[Theorem~4.1]{Gre06} modulo deriving the necessary positivity.

\begin{proof}[Proof of \cref{thm:main-functional}]
By replacing $M_1D$ by $D$ (note that $M_1$ is invertible) we can reduce the case of $(M_1,M_2)$ to the case of $(I,M_2M_1^{-1})$. Hence from now on we assume that $M_1 = I$ and $M_2 = J$. Applying \cref{arith-reg-lem}, we decompose 
\[f = f_{\on{str}}+f_{\on{sml}}+f_{\on{psr}}\]
where $f_{\on{str}} = \mb{E}[f|\mf{B}]$ for a $k$-symmetrized quadratic factor $\mf{B}$ of complexity $(d_1,d_2,d_3)$ and rank at least $\omega_1(d_1+d_2+d_3)$, where $\snorm{f_{\on{sml}}}_2\le \epsilon/250$, and where $\snorm{f_{\on{psr}}}_{U^3(G^k)}\le 1/\omega_2(d_1+d_2+d_3)$. The complexity $(d_1,d_2,d_3)$ is bounded in terms of the given parameters and growth functions. We define $H$ to be the subspace of $(\mb{F}_p^n)^k$ such that the linear factors of $\mf{B}$ are zero. We will prove that 
\begin{equation}\label{eq:key-average}
\mb{E}_{X,D\in G^k}[f(X)f(X+D)f(X+JD)f(X+(J+I)D)\mbm{1}_H(D)]\ge p^{-kd_1} (\alpha^4-\epsilon).
\end{equation}
This suffices to prove the result as we are summing over a density $p^{-kd_1}$ subset of the differences $D$ and hence at least one difference achieves the necessary bound. Furthermore, by Markov's inequality, a fraction of at least $\Omega_{\alpha,\epsilon}(1)$ of the differences in $H$ satisfy the weaker bound of $\alpha^4-2\epsilon$. Adjusting $\epsilon$ appropriately will prove the desired upon noting that a positive fraction of $G^k$ lies in $H$ due to the bounded complexity of $\mf{B}$.

Now we focus attention on \cref{eq:key-average}. Expanding $f = f_{\on{str}}+f_{\on{sml}}+f_{\on{psr}}$ allows us to turn the left side into $81$ terms of the form 
\[\mb{E}_{X,D\in G^k}[f_1(X)f_2(X+D)f_3(X+JD)f_4(X+(J+I)D)\mathbbm{1}_H(D)]\]
where $f_1,f_2,f_3,f_3\in\{f_{\on{str}},f_{\on{sml}},f_{\on{psr}}\}$. If $f_{\on{sml}}$ appears in the expression, we have that 
\begin{align*}
|\mb{E}_{X,D}[f_1(X)f_2(X+D)&f_3(X+JD)f_4(X+(J+I)D)\mathbbm{1}_H(D)]|\\
&\le\mb{E}_X[|f_{\on{sml}}(X)|\mathbbm{1}_H(D)] \le p^{-kd_1}(\mb{E}_X[f_{\on{sml}}(X)^2])^{1/2}\le p^{-kd_1}\epsilon/250.
\end{align*}
This bounds the 65 terms that include $f_{\on{sml}}$.

Next we bound the terms that include $f_{\on{psr}}$. Say $f_3=f_{\on{psr}}$ (the other cases are analogous). Note that we have 
\[\mbm{1}_H(D) = \sum_{T\in H^{\perp}}\mbm{1}_{T+H}(X+D)\mbm{1}_{T+H}(X)\]
for all $X$. Therefore 
\begin{align*}
|\mb{E}[f_1(X)&f_2(X+D)f_{\on{psr}}(X+JD)f_4(X+(I+J)D)\mathbbm{1}_H(D)]|\\
&=\left|\sum_{T\in H^{\perp}}\mb{E}[f_1(X)f_2(X+D)f_{\on{psr}}(X+JD)f_4(X+(I+J)D)\mbm{1}_{T+H}(X+D)\mbm{1}_{T+H}(X)]\right|\\
&\le |H^{\perp}|\snorm{f_{\on{psr}}}_{U^3(G^k)}
\end{align*}
where in the final line we have used \cref{lem:gowers-inequality} (since $I,J,I+J,I-J$ are invertible). Taking the growth function $\omega_2$ to be a sufficiently fast growing exponential (depending on $p,k$), we can ensure that $\|f_{\on{psr}}\|_{U^3(G^k)}\leq p^{-2kd_1}\epsilon/250$.

Thus it suffices to prove
\begin{equation}\label{eq:count-inequality}
\mb{E}[f_{\on{str}}(X)f_{\on{str}}(X+D)f_{\on{str}}(X+JD)f_{\on{str}}(X+(J+I)D)\mathbbm{1}_H(D)]\ge p^{-kd_1}(\alpha^4-\epsilon/2).
\end{equation}

Since $f_{\on{str}} = \mb{E}[f|\mf{B}]$, we see that $f_{\on{str}}$ is $\mf{B}$-measurable and $[0,1]$-valued. Recall that a factor $\mf{B}$ defines a $\sB\colon G^k = (\mb{F}_p^n)^k \to (\mb{F}_p^k)^{d_1}\times(\mc{S}_k)^{d_2}\times (\mc{S}'_k)^{d_3}$. Hence there is an associated function $\mbf{f}\colon (\mb{F}_p^k)^{d_1}\times(\mc{S}_k)^{d_2}\times (\mc{S}'_k)^{d_3}\to[0,1]$ such that $f_{\on{str}}(X) = \mbf{f}(\msf{B}(X))$.

\begin{claim}
\label{thm:str-count}
\[\mb{E}[f_{\on{str}}(X)f_{\on{str}}(X+D)f_{\on{str}}(X+JD)f_{\on{str}}(X+(J+I)D)\mathbbm{1}_H(D)]\]
and
\[p^{-kd_1}\mb{E}_{\vec v\in\F_p^k,\vec M_2\in(\Lambda_J^\perp)^{d_2},\vec M_3\in(\Lambda_J'^\perp)^{d_3}} [\mbf{f}(\vec{v},M_2^{(1)},M_3^{(1)})\mbf{f}(\vec{v},M_2^{(2)},M_3^{(2)})\mbf{f}(\vec{v},M_2^{(3)},M_3^{(3)})\mbf{f}(\vec{v},M_2^{(4)},M_3^{(4)})]\]
are equal to up to a multiplicative factor of $1+O(p^{-r/2+2kd_1+(2\binom{k+1}{2}+k^2)d_2+(2\binom{k}{2}+k^2)d_3})$ multiplicative factor, where the rank $r$ is the rank of $\fB$.
\end{claim}

\begin{proof}
Restricting $D$ to lie in $H$ we need to understand the equidistribution of $(\msf{B}(X),\msf{B}(X+D),\msf{B}(X+JD),\msf{B}(X+(I+J)D))$. We apply \cref{thm:counting}. We have that $(\msf{B}_1(X),\msf{B}_1(X+D),\msf{B}_1(X+JD),\msf{B}_1(X+(I+J)D))$ equidistributes over $(\vec{v},\vec{v},\vec{v},\vec{v})$ with $\vec{v}\in(\mb{F}_p^k)^{d_1}$, since looking at $D\in H$ corresponds precisely to looking at atoms for which linear factors in $\mf{B}$ for $\msf{B}(X),\msf{B}(X+D)$ are equal. (We are also implicitly using that $D\in H$ implies each row of $D$ is in the orthogonal space of the defining vectors $\{r_1,\ldots,r_{d_1}\}$ of the linear factors of $\mf{B}$, which implies $JD\in H$ as well.)

Furthermore
\[(M_2^{(1)},M_2^{(2)},M_2^{(3)},M_2^{(4)}) = (\msf{B}_2(X),\msf{B}_2(X+D),\msf{B}_2(X+JD),\msf{B}_2(X+(I+J)D))\] equidistributes over $(\Lambda_J^\perp)^{d_2}$ and finally \[(M_3^{(1)},M_3^{(2)},M_3^{(3)},M_3^{(4)}) = (\msf{B}_3(X),\msf{B}_3(X+D),\msf{B}_3(X+JD),\msf{B}_3(X+(I+J)D))\]
equidistributes over $(\Lambda_J'^\perp)^{d_3}$ with each of the components distributing independently. Therefore, even upon restricting the image of $\msf{B}_1$ to be zero, we have equidistribution over the remaining space.
\end{proof}

It remains to study the expectation in \cref{thm:str-count}. Recalling the definitions of $\Xi_J, \Lambda_J, \Lambda_J'$ from \cref{sec:equidistribution}, we define the following subspaces:
\begin{align*}
\Omega_J &= \{(-A,-A(I+J)(I-J)^{-1})\colon A^\TT = +A, A\in\Xi_J\},\\
\Omega_J' &= \{(-A,-A(I+J)(I-J)^{-1})\colon A^\TT = -A, A\in\Xi_J\}.
\end{align*}

\begin{claim}
\label{thm:lambda-structure}
Given $(M^{(1)},M^{(2)},M^{(3)},M^{(4)})\in\mc{S}_k^4$, we have $(M^{(1)},M^{(2)},M^{(3)},M^{(4)})\in\Lambda_J^\perp$ if and only if we have the equality of cosets \[(M^{(1)},M^{(2)})+\Omega_J^\perp=(M^{(4)},M^{(3)})+\Omega_J^\perp.\]
Similarly given $(M^{(1)},M^{(2)},M^{(3)},M^{(4)})\in\mc{S}_k'^4$, we have $(M^{(1)},M^{(2)},M^{(3)},M^{(4)})\in\Lambda_J'^\perp$ if and only if we have the equality of cosets \[(M^{(1)},M^{(2)})+\Omega_J'^\perp=(M^{(4)},M^{(3)})+\Omega_J'^\perp.\]
\end{claim}

This claim follows by inspection of the definitions of $\Lambda_J$ and $\Omega_J$. From this we deduce
\begin{align*}
\mb{E}[\mbf{f}(\vec{v},M_2^{(1)},&M_3^{(1)})\mbf{f}(\vec{v},M_2^{(2)},M_3^{(2)})\mbf{f}(\vec{v},M_2^{(3)},M_3^{(3)})\mbf{f}(\vec{v},M_2^{(4)},M_3^{(4)})]\\
&= \mb{E}_{\vec{v}}\mb{E}_{\tau}\left(\mb{E}_{ \big((M_2^{(1)},M_2^{(2)})+\Omega_J^{\perp},(M_3^{(1)},M_3^{(2)})+\Omega_J'^{\perp}\big)=\tau} \mbf{f}(\vec{v},M_2^{(1)},M_3^{(1)})\mbf{f}(\vec{v},M_2^{(2)},M_3^{(2)})\right)^2\\
&\ge\left(\mb{E}_{\vec{v}}\mb{E}_\tau\mb{E}_{ \big((M_2^{(1)},M_2^{(2)})+\Omega_J^{\perp},(M_3^{(1)},M_3^{(2)})+\Omega_J'^{\perp}\big)=\tau} \mbf{f}(\vec{v},M_2^{(1)},M_3^{(1)})\mbf{f}(\vec{v},M_2^{(2)},M_3^{(2)})\right)^2\\
&= \left(\mb{E}_{\vec{v}}\mb{E}_{(M_2^{(1)},M_2^{(2)},M_3^{(1)},M_3^{(2)})} \mbf{f}(\vec{v},M_2^{(1)},M_3^{(1)})\mbf{f}(\vec{v},M_2^{(2)},M_3^{(2)})\right)^2\\
&= \mb{E}_{\vec{v}}\left(\mb{E}_{(M_2^{(1)},M_3^{(1)})} \mbf{f}(\vec{v},M_2^{(1)},M_3^{(1)})^2\right)^2\\
&\ge \left(\mb{E}_{\vec{v},(M_2^{(1)},M_3^{(1)})} \mbf{f}(\vec{v},M_2^{(1)},M_3^{(1)})\right)^4
\end{align*}
where we have used the Cauchy--Schwarz inequality twice. To finish note that by the equidistibution derived in \cref{prop:factor-equidistribution} we have that $\mb{E}_{\vec{v},(M_2^{(1)},M_2^{(2)})} \mbf{f}(\vec{v},M_2^{(1)},M_3^{(1)})$ is $\mb{E}[f_{\on{str}}]$ up to a multiplicative factor of $1+O(p^{-r/2+kd_1+\binom{k+1}{2}d_2+\binom{k}{2}d_3})$. Finally since $\mb{E}[f_{\on{str}}] = \mb{E}[\mb{E}[f|\mf{B}]] = \alpha$ the desired result follows upon taking the growth function $\omega_2$ (and thus $r$, the rank of $\fB$) to be large enough.
\end{proof}

\begin{remark}
The key reason that this argument works is a ``positivity'' result in the final expectation. This ``positivity'' occurs from the symmetry of $\Lambda_J^{\perp}$ which ultimately is derived from the spectral condition on $J$. In the next section we consider what occurs when the spectral condition on $J$ is no longer satisfied. In essence, $\Lambda_J^{\perp}$ will be one dimension larger than otherwise in such a way that it no longer has this symmetry property.
\end{remark}

\section{Counterexample to popular differences for rotated squares in \texorpdfstring{$\mb{F}_5^n$}{F5n}}\label{sec:counterexample}
Given the results of the last section it is natural to ask whether the popular difference result holds for all matrix patterns which are controlled by the $U^3$-norm. We now prove that this is false and demonstrate \cref{thm:counterexample}.

As is standard, it suffices to give a construction for a set of fixed positive density, as smaller sets will follow from subsampling and simple concentration facts. In fact, we will merely give a function in $[0,1]$ (which one can scale down appropriately and sample from). The counterexample proceeds in stages. In the first stage we give a construction which rules out all sufficiently generic differences (i.e., $(a,b)$ such that $\on{span}_{\mb{F}_5}\{a,b\}$ is $2$-dimensional), using the failure of the key positivity in \cref{sec:equidistribution}. We then modify the function to rule out the non-generic directions using techniques from earlier work of the second and third authors with Zhao in \cite[Sections~2,~3]{SSZ20} (which deals with, for instance, axis-aligned squares) with \cite[Section~3]{SSZ20} itself building on a construction of Mandache \cite{Man18}. Finally, we expect that the construction here can be used to disprove the ergodic analogue \cite[Question~1.11]{ABB21} but we do not pursue this direction here.

First, if $\Gamma$ is the linear automorphism of $(\mb{F}_5^n)^2$ defined by $(x,y)\mapsto(x-2y,x+2y)$ then by replacing $A$ by $\Gamma^{-1}A$ and reparametrizing the pattern we are counting, it suffices to consider the ``diagonalized'' pattern
\[(x,y),(x+a,y+b),(x+2a,y-2b),(x+3a,y-b).\]
(Note that this is two arithmetic progressions in the coordinates with a twist, though we will not make use of this fact)

\subsection{Initial construction}\label{sub:initial-construction}
Let $f_1\colon (\mb{F}_5^n)^2\to[0,1]$ be defined by
\[f_1(x,y) = g_1(x\cdot x, x\cdot y),\]
where $g_1\colon \mb{F}_5^2\to[0,1]$ is to be chosen later. With this choice, the pattern count for $f$ depends only on $g_1$ and the distribution of 8-tuples of the form
\begin{equation*}
(x\cdot x, x\cdot y, \ldots, (x+3a)\cdot(x+3a), (x+3a)\cdot(y-b)).
\end{equation*}

Let $\Lambda_2'\le\mb{F}_5^8$ be the space orthogonal to the vectors
\[(1,0,-1,0,-1,0,1,0),\quad(0,1,0,-1,0,-1,0,1),\quad(1,0,-3,0,3,0,-1,0)\]
We claim that these 8-tuples equidistribute over $\Lambda_2'$.

\begin{proposition}\label{prop:counterexample-equidistribution}
Fix nonzero $a,b\in\mb{F}_5^n$ that are not multiples of each other. Then
\begin{align*}(x\cdot x, x\cdot y, (x+a)\cdot(x+a), &(x+a)\cdot(y+b), (x+2a)\cdot(x+2a),\\
&(x+2a)\cdot(y-2b), (x+3a)\cdot(x+3a), (x+3a)\cdot(y-b))\end{align*}
obtains every value in $(0,0,a\cdot a, a\cdot b,4a\cdot a,-4a\cdot b,9a\cdot a,-3a\cdot b)+\Lambda_2'$ with probability
\[5^{-5} + O(5^{-n/2})\]
\end{proposition}
\begin{remark}
Note that $(0,0,0,1,0,-4,0,-3)\in\Lambda_2'$, so it actually equidistributes in $(0,0,a\cdot a, 0,4a\cdot a,0,9a\cdot a,0)+\Lambda_2'$.
\end{remark}
\begin{proof}
This is a hands-on computation of equidistribution. Apply  \cref{prop:equidistribution} to the concatenation $x' = (x,y)$ with linear forms $x \cdot a$, $x \cdot b$, $a \cdot y$ and quadratic forms $x \cdot x$, $x \cdot y$ (the last of which can be written as a symmetric matrix when $p \neq 2$). We conclude that the image of $(x,y)$ under this 5-tuple of forms obtains each point in $\F_5^5$ with probability $5^{-5}+O(5^{-n/2}).$ Given these five values one can solve for all values in the 8-tuple, and solving we obtain equidistribution of the 8-tuple over the appropriate 5-dimensional subspace $\Lambda_2'$ as desired.
\end{proof}
Consequently, if $a,b\in\mb{F}_5^n$ are nonzero and not multiples of each other,
\begin{align}
\beta_1&(a,b) := \mb{E}f_1(x,y)f_1(x+a,y+b)f_1(x+2a,y-2b)f_1(x+3a,y-b)\notag\\
&= \mb{E}g_1(x\cdot x,x\cdot y)g_1((x+a)\cdot(x+a),(x+a)\cdot(y+b))g_1((x+2a)\cdot(x+2a),(x+2a)\cdot(y-2b))\notag\\
&\qquad\qquad g_1((x+3a)\cdot(x+3a),(x+3a)\cdot(y-b))\notag\\
&= \mb{E}_{v\in\Lambda_2'}g_1(v_1,v_2)g_1(v_3+a\cdot a,v_4)g_1(v_5+4a\cdot a,v_6)g_1(v_7+9a\cdot a,v_8) + O(5^{-n/2}),\label{eq:counterexample-g}
\end{align}
using the remark after \cref{prop:counterexample-equidistribution} in the last line.

Let $g_1(x,y) = \mbm{1}_S(x,y)$ be the indicator of the set
\[S = \{(0,2),(0,3),(0,4),(1,0),(1,3),(1,4),(2,1),(2,2),(3,0),(3,1)\}.\]
Then a direct verification proves that 
\[\sup_{a\cdot a\in \mb{F}_5}\mb{E}_{v\in\Lambda_2'}g_1(v_1,v_2)g_1(v_3+a\cdot a,v_4)g_1(v_5+4a\cdot a,v_6)g_1(v_7+9a\cdot a,v_8) = \frac{73}{5^5}\]
while clearly
\[\alpha_1 = \mb{E}f_1(x,y) = \mb{E}_{v_1,v_2\in\mb{F}_5}g_1(v_1,v_2) + O(5^{-n/2}) = \frac{2}{5} + O(5^{-n/2})\]
by \cref{prop:counterexample-equidistribution}. (Code verifying this explicit finite computation is attached on the arXiv.)

If we merely wish to establish \cref{thm:counterexample} for all but $O(\sqrt{|G|})$ directions, where $G = (\mb{F}_5^n)^2$, then we are done due to $73/5^5 < (2/5)^4$. Now we introduce further constructions which allow us to fix the directions where $a,b$ are nontrivially related.

\subsection{Fixing most special directions}\label{sub:fixing-construction}
We next proceed with a modification of a construction which appears in \cite{Man18} and was extended in \cite{SSZ20}. Let $X_x,Y_y,Z_z,X_x',Y_y',Z_z'$ for $x,y,z\in\mb{F}_5^n$ be independent random variables uniform on $[0,1]$. Let $F_2, F_3\colon (\mb{F}_5^n)^2\to[0,1]$ be defined by
\[F_2(x,y) = g_2(X_{-x-y},Y_{-2x+2y},Z_{2x+y}),\quad F_3(x,y) = g_2(X_{-x-2y}',Y_{-2x-y}',Z_{2x+2y}')\]
where $g_2\colon [0,1]^3\to[0,1]$ is a function to be chosen later. This is a random function. Let $h(x,y) = f_1(x,y)F_2(x,y)F_3(x,y)$, which is also a random function. Let
\[\alpha_h = \mb{E}_{x,y\in\mb{F}_5^n}h(x,y)\]
and let
\[\beta_h(a,b) = \mb{E}_{x,y\in\mb{F}_5^n}h(x,y)h(x+a,y+b)h(x+2a,y-2b)h(x+3a,y-b),\]
which are both random in the $\mbf{X},\mbf{Y},\mbf{Z},\mbf{X}',\mbf{Y}',\mbf{Z}'$ variables. We have
\begin{align*}
\alpha_2 &= \mb{E}_{\mbf{X},\mbf{Y},\mbf{Z}}\alpha_h = \mb{E}_{x,y\in\mb{F}_5^n}\mb{E}_{\substack{\mbf{X},\mbf{Y},\mbf{Z}\\\mbf{X}',\mbf{Y}',\mbf{Z}'}}f_1(x,y)g_2(X_{-x-y},Y_{-2x+2y},Z_{2x+y})g_2(X_{-x-y}',Y_{-2x+2y}',Z_{2x+y}')\\
&= \mb{E}_{x,y\in\mb{F}_5^n}f_1(x,y)\cdot(\mb{E}_{u,v,w\in[0,1]}g_2(u,v,w))^2
\end{align*}
since $\mbf{X},\mbf{Y},\mbf{Z}$ are independent, etc. We also have
\begin{align*}
\beta_2(a,b) &= \mb{E}_{\mbf{X},\mbf{Y},\mbf{Z}}\beta_h(a,b)\\
&= \mb{E}_{x,y\in\mb{F}_5^n}\mb{E}_{\substack{\mbf{X},\mbf{Y},\mbf{Z}\\\mbf{X}',\mbf{Y}',\mbf{Z}'}}f_1(x,y)f_1(x+a,y+b)f_1(x+2a,y-2b)f_1(x+3a,y-b)\\
&\qquad g_2(X_{-x-y},Y_{-2x+2y},Z_{2x+y})g_2(X_{-x-y-a-b},Y_{-2x+2y-2a+2b},Z_{2x+y+2a+b})\\
&\qquad g_2(X_{-x-y-2a+2b},Y_{-2x+2y+a+b},Z_{2x+y-a-2b})g_2(X_{-x-y+2a+b},Y_{-2x+2y-a-2b},Z_{2x+y+a-b})\\
&\qquad g_2(X_{-x-2y}',Y_{-2x-y}',Z_{2x+2y}')g_2(X_{-x-2y-a-2b}',Y_{-2x-y-2a-b}',Z_{2x+2y+2a+2b}')\\
&\qquad g_2(X_{-x-2y-2a-b}',Y_{-2x-y+a+2b}',Z_{2x+2y-a+b}')g_2(X_{-x-2y+2a+2b}',Y_{-2x-y-a+b}',Z_{2x+2y+a-2b}').
\end{align*}
If $a,b\in\mb{F}_5^n$ are not linearly dependent, then we easily see that all the terms in the product are independent, and linearity of expectation demonstrates
\begin{align*}
\beta(a,b) &= \mb{E}_{x,y\in\mb{F}_5^n}f_1(x,y)f_1(x+a,y+b)f_1(x+2a,y-2b)f_1(x+3a,y-b)(\mb{E}g_2(u,v,w))^8\\
&= \beta_1(a,b)(\mb{E}g_2(u,v,w))^8.
\end{align*}
Otherwise we have cases depending on the $6$ possible linear dependencies. The key point is that, for example,
\[\mb{E}_{\mbf{X},\mbf{Y},\mbf{Z}}g_2(X_0,Y_0,Z_0)g_2(X_0,Y_1,Z_1) = \mb{E}_{u_0,v_0,w_0,v_1,w_1\in[0,1]}g_2(u_0,v_0,w_0)g_2(u_0,v_1,w_1),\]
so in each of these $6$ cases we can reduce to some sort of expectation of $g_2$, or rather, the product of two such terms coming from the independent sets of variables $(\mbf{X},\mbf{Y},\mbf{Z})$ and $(\mbf{X}',\mbf{Y}',\mbf{Z}')$.

Explicit computation yields
\begin{align*}
\beta_2(a,0) &= \beta_1(a,0)(\mb{E}g_2(u_0,v_0,w_0)g_2(u_1,v_1,w_1)g_2(u_2,v_2,w_2)g_2(u_3,v_3,w_3))^2\\
&= \beta_1(a,0)(\mb{E}g_2(u,v,w))^8,\\
\beta_2(a,a) &= \beta_1(a,a)(\mb{E}g_2(u_0,v_0,w_0)g_2(u_1,v_0,w_1)g_2(u_0,v_2,w_2)g_2(u_1,v_2,w_0))\\
&\qquad\qquad\qquad(\mb{E}g_2(u_0,v_0,w_0)g_2(u_1,v_1,w_1)g_2(u_1,v_2,w_0)g_2(u_3,v_0,w_1)),\\
\beta_2(a,-a) &= \beta_1(a,-a)(\mb{E}g_2(u_0,v_0,w_0)g_2(u_0,v_1,w_1)g_2(u_2,v_0,w_1)g_2(u_2,v_1,w_3))\\
&\qquad\qquad\qquad(\mb{E}g_2(u_0,v_0,w_0)g_2(u_1,v_1,w_0)g_2(u_2,v_1,w_2)g_2(u_0,v_3,w_2)),\\
\beta_2(a,2a) &= \beta_1(a,2a)(\mb{E}g_2(u_0,v_0,w_0)g_2(u_1,v_1,w_1)g_2(u_1,v_2,w_0)g_2(u_3,v_0,w_1))\\
&\qquad\qquad\qquad(\mb{E}g_2(u_0,v_0,w_0)g_2(u_0,v_1,w_1)g_2(u_2,v_0,w_1)g_2(u_2,v_1,w_3)),\\
\beta_2(a,-2a) &= \beta_1(a,-2a)(\mb{E}g_2(u_0,v_0,w_0)g_2(u_1,v_1,w_0)g_2(u_2,v_1,w_2)g_2(u_0,v_3,w_2))\\
&\qquad\qquad\qquad(\mb{E}g_2(u_0,v_0,w_0)g_2(u_1,v_0,w_1)g_2(u_0,v_2,w_2)g_2(u_1,v_2,w_0)),\\
\beta_2(0,b) &= \beta_1(0,b)(\mb{E}g_2(u_0,v_0,w_0)g_2(u_1,v_1,w_1)g_2(u_2,v_2,w_2)g_2(u_3,v_3,w_3))^2\\
&= \beta_1(0,b)(\mb{E}g_2(u,v,w))^8.
\end{align*}
Here all variables other than $a,b\in\mb{F}_5^n$ are uniform on $[0,1]$. It is worth noting every complicated product of expectations for the middle four terms are very similar. In particular, they are all equal to the product of the densities of two tripartite $3$-uniform hypergraphs in the symmetric tripartite hypergraphon $g_2$ (with embeddings respecting the tripartition). Furthermore, the two hypergraphs attained for each term is the same pair of hypergraphs if we disregard the data of the tripartition. For the function $g_2$ we will choose, these considerations will not affect the bounds we prove, so we will focus on a single term.

We first define $g_2$. Let $L\ge 1$ and let $\Lambda$ be a subset of $\mb{Z}/L\mb{Z}$ avoiding arithmetic progressions of length $3$ of size $L\exp(-C\sqrt{\log L})$ for some absolute constant $C > 0$. Let $U = V = W = \mb{Z}/L\mb{Z}$ and $H$ be a tripartite graph on $U\times V\times W$ with $(s,s+t)\in U\times V$ an edge when $t\in\Lambda$, $(s,s+t)\in V\times W$ an edge when $t\in\Lambda$, and $(s,s+2t)\in U\times W$ an edge when $t\in\Lambda$. Note that the triangles in $H$ are of the form $(s,s+t,s+2t)$ for $t\in\Lambda$ precisely since $\Lambda$ has no nontrivial arithmetic progressions of length $3$. This has the special property that every edge is in a unique triangle of $H$ (this is the Ruzsa-Szemer\'edi graph).

Now let $g_2(u,v,w) = 1$ if $(\lfloor Lu\rfloor,\lfloor Lv\rfloor,\lfloor Lw\rfloor)\pmod{L}$ encodes one of these triangles. We have
\[\mb{E}g_2(u,v,w) = \frac{L|\Lambda|}{L^3} = L^{-1}\exp(-C\sqrt{\log L}).\]
Now we find
\[\mb{E}g_2(u_0,v_0,w_0)g_2(u_1,v_0,w_1)g_2(u_0,v_2,w_2)g_2(u_1,v_2,w_0).\]
When we refer to $u\in[0,1]$ in the following discussion, we mean $\lfloor Lu\rfloor\pmod{L}$. In order for a term to be $1$, we must have $u_0v_0w_0$, $u_1v_0w_1$, $u_0v_2w_2$, $u_1v_2w_0$ be triangles. But then $u_0w_0,u_0v_2,v_2w_0$ are all edges in $H$, so $u_0v_2w_0$ is also a triangle. This forces $v_2 = v_0$ (as edges are in unique triangles). Then $u_1v_0w_0$ is a triangle, so $u_1 = u_0$. Finally, this means $u_0v_0w_1$ is a triangle so $w_0 = w_1$. Thus the expectation equals
\[L^{-4}\mb{E}g_2(u,v,w) = \frac{L|\Lambda|}{L^7} = L^{-5}\exp(-C\sqrt{\log L}).\]
Next, we find
\[\mb{E}g_2(u_0,v_0,w_0)g_2(u_1,v_1,w_1)g_2(u_1,v_2,w_0)g_2(u_3,v_0,w_1)\le\frac{L^4}{L^8} = L^{-4}\]
since there are at most $L^4$ ways to choose $w_0,u_1,w_1,v_0$. After that, the variables $u_0,v_1,v_2,u_3$ are forced in order to guarantee a nonzero term, since edges are in unique triangles. This demonstrates that
\[\beta_2(a,\lambda a)\le\beta_1(a,\lambda a)L^{-9}\exp(-C\sqrt{\log L})\le L^{-9+o(1)}\]
for $\lambda\in\mb{F}_5^\ast$. Now we take $L$ sufficiently large. First, if $a,b\in\mb{F}_5^n$ are linearly independent, then by \cref{sub:initial-construction} and the above we have
\[\beta_2(a,b) = \beta_1(a,b)(\mb{E}g_2(u,v,w))^8\le\bigg(\frac{73}{80}+O(5^{-n/2})\bigg)\alpha_2^4.\]
On the other hand, if $b = \lambda a$ for $\lambda\in\mb{F}_5^\ast$ then
\[\beta_2(a,b)\le L^{-9+o(1)}\le\frac{1}{2}\alpha_2^4\]
since $\alpha_2$ grows as $L^{-2+o(1)}$ times $\mb{E}f_1(x,y) = 2/5+O(5^{-n/2})$.

Therefore, in expectation our random function satisfies the conclusion of \cref{thm:counterexample} for $(a,b)\in(\mb{F}_5^n\setminus 0)^2$. To obtain a single function which satisfies all these inequalities (of which there are $O(|G|)$), we use a concentration argument. A straightforward extension of \cite[Lemma~3.2]{SSZ20} suffices. Roughly note that each of the above quantities can be controlled using Azuma--Hoeffding inequality on the Doob-martingale where one successively conditions on each of the random variables $X_x,Y_y,Z_z,X_x',Y_y',Z_z'$ for $x,y,z\in\mb{F}_5^n$ as each random variables only participates in a small number of terms. The quality of concentration is at least $\exp(-|G|^c)$ for deviations on the scale of $1/|G|^{1/4}$, which is more than sufficient. Let $h(x,y)$ now denote a specific instantiation of the above defined random function which satisfies
\[\sup_{(a,b)\in (\mb{F}_5^n\setminus 0)^2} \mb{E}[h(x,y)h(x+a,y+b)h(x+2a,y-2b)h(x+3a,y-b)]\le (1-\delta)\mb{E}[h(x,y)]^4\]
for an absolute constant $\delta>0$.

\subsection{Finishing the construction}\label{sub:finishing-construction}
Finally we are in position to fix the directions $(a,b)$ where $a = 0$ or $b = 0$. We now use a randomized version of the construction from \cite[Section~2]{SSZ20} in order to eliminate these final special differences. Define $T = \{0,1,2\}^{\gamma}\times (\mb{F}_5)^{n-\gamma}$, where $\gamma\ge 1$ is an integer to be chosen later. Note that this set has density $\beta = (3/5)^{\gamma}$ but a four term arithmetic progression density of $(3/25)^{\gamma}\le \beta^{4.15}$ noting that the set $\{0,1,2\}$ has no nontrivial four term arithmetic progressions. Now for each $g\in\mb{F}_5^n$ choose a two uniformly random bijective affine transformations $\phi(g),\phi'(g)$ of $\mb{F}_5^n$. Then let
\[f(x,y) = h(x,y)\mbm{1}_{x\in\phi(y)S}\mbm{1}_{y\in\phi'(x)S}.\]
First, we have 
\[\mb{E}_{\phi,\phi'}[\mb{E}_{x,y}f(x,y)] = \mb{E}_{x,y}[h(x,y)\mb{E}_{\phi,\phi'}[\mbm{1}_{x\in\phi(y)S}\mbm{1}_{y\in\phi'(x)S}]] = \beta^2\mb{E}[h(x,y)].\]
Similarly for $(a,b)$ with nonzero coordinates we find that 
\begin{align*}
\mb{E}_{\phi,\phi'}&[\mb{E}_{x,y}f(x,y)f(x+a,y+b)f(x+2a,y-2b)f(x+3a,y-b)]\\ &=\beta^8\mb{E}_{x,y}[h(x,y)h(x+a,y+b)h(x+2a,y-2b)h(x+3a,y-b)]\\
&\le(1-\delta)\beta^8(\mb{E}h(x,y))^4.
\end{align*}
Finally assume that exactly one of $a$ or $b$ is zero. We handle the case when $a$ is zero as the other is analogous. We find
\begin{align*}
\mb{E}&_{\phi,\phi'}[\mb{E}_{x,y}f(x,y)f(x,y+b)f(x,y-2b)f(x,y-b)]\\
&\le\mb{E}_{\phi,\phi',x,y}[\mbm{1}_{x\in\phi(y)S}\mbm{1}_{y\in\phi'(x)S}\mbm{1}_{x\in \phi(y+b)S}\mbm{1}_{y+b\in\phi'(x)S}\mbm{1}_{x\in\phi(y-2b)S}\mbm{1}_{y-2b\in\phi'(x)S}\mbm{1}_{x\in\phi(y-b)S}\mbm{1}_{y-b\in\phi'(x)S}]\\
&= \beta^4\mb{E}_{\phi',x,y}[\mbm{1}_{y\in\phi'(x)S}\mbm{1}_{y+b\in \phi'(x)S}\mbm{1}_{y-2b\in\phi'(x)S}\mbm{1}_{y-b\in\phi'(x)S}]\\
&\le\beta^{8.15}.
\end{align*}
The last line follows since for every random map $\phi'(x)$, there is at most a density of $\beta^{4.15}$ of four term arithmetic progressions in $\phi'(x)S$. Therefore taking $\gamma$ to be a sufficiently large multiple of $\log(1/\mb{E}[h(x,y)])$ (which is of constant order due to \cref{sub:fixing-construction}) guarantees that each of these expectations is at most $(1-\delta)\beta^8(\mb{E}h(x,y))^4$. Furthermore, each of the above densities concentrates (with respect to the randomness of $\phi,\phi'$) with high probability by Azuma--Hoeffding, so \cref{thm:counterexample} follows.

\section{Popular differences for \texorpdfstring{$3$}{3}-point patterns}\label{sec:three-point}
We end by proving popularity for all admissible three-point patterns, namely, \cref{thm:3ptsZ}. As is standard, we actually prove an analogous result over all compact abelian groups. We deduce the version over $\mb{Z}$ from results (mod $N$) using a trick of Green \cite{Gre05}.

These patterns can be handled in a direct Fourier-analytic manner. The proof here is closely modeled after the proof for $3$-term arithmetic progressions. For a compact abelian group $G$ with Haar measure $\mu$, finite $S\subseteq\wh{G}$, and $\delta > 0$, define the Bohr set
\[B(S,\delta) = \{x \in G: \max_{\xi\in S}(\|\xi x\|_{\mb{R}/\mb{Z}}) < \delta\}.\]
We recall the standard fact that $\mu(B(S,\delta)) = \Omega_{|S|,\delta}(1)$ (see \cite[Lemma~4.20]{TV06}). We will further write for a Bohr set $B$ that $\mu_B$ is the uniform measure on it obtained by restricting the Haar measure appropriately.
\begin{theorem}\label{thm:3ptPatterns}
For any $\alpha, \epsilon>0$ there exists $C(\epsilon) > 0$ so that the following holds. Let $G$ be a compact abelian group with Haar probability measure $\mu$. Let $M_1, M_2$ be continuous automorphisms of $G$ such that $M_1-M_2$ is an automorphism. Then for any function $f\colon G\to [0,1]$ with $\int_{g\in G}f(g)\mu(g) \ge \alpha$, there is a measure $\nu_D$ with $\|\nu_D\|_{L^\infty} \le C$ such that
\begin{equation}\label{eq:3ptPatterns}
\int_{x\in G,d\in G}f(x)f(x+M_1d)f(x+M_2d)\nu_D(d)\mu(x)\ge\alpha^3 - \epsilon.
\end{equation}
Moreover, given a Bohr set $B_0 = B(S_0,\rho_0)$, one may take $\nu_D = \mu_B \ast \mu_B$, where $B =  B(S,\rho)$ with $S \supseteq S_0$ and $\rho\le \rho_0$. In this case, the absolute constant $C$ will also depend on $|S_0|, \rho_0$.
\end{theorem}

From this general statement one can easily obtain popular difference results in their standard forms for finite abelian groups (\cref{thm:3ptsFinite}) and for $\mb{Z}$ (\cref{thm:3ptsZ}). The latter answers (as a special case) \cite[Question~1.16]{ABB21} and is the combinatorial analog of \cite[Theorem~1.10]{ABB21}. As much of the proof is identical to that of the three-term arithmetic progression case in \cite{TaoBlog}, we collect the necessary results in the following proposition. Essentially we are extracting a statement of the strong regularity lemma from \cite{TaoBlog}. Closely related statements appear in \cite{Bou86, Gre06}. 
\begin{proposition}[{\cite{TaoBlog}}]
Fix a compact abelian group $G$, parameters $\delta, \epsilon > 0$, and a set $S_0 \subseteq\wh{G}$, as well as growth functions $\omega_1,\omega_2\colon\mb{R}^+\to\mb{R}^+$. Given a function $f\colon G\to [0,1]$, there are $\gamma_1,\gamma_2 > 0$ and a finite set $T$ satisfying $S_0\subseteq T \subseteq\wh{G}$ with the following properties.
\begin{enumerate}[\quad1.]
	\item $|T| = O_{\delta,\epsilon,|S_0|,\omega_1,\omega_2}(1)$.
	\item $\gamma_1\le1/\omega_1(|T|+\delta^{-1}+\epsilon^{-1})$ and $\gamma_2\le1/\omega_2(\gamma_1^{-1})$, whereas $\gamma_2$ is bounded away from $0$ independent of $f$.
	\item We have the decomposition $f = f_1 + f_2 + f_3$, where
	\begin{enumerate}
		\item $f_1, f_2, f_3$ are $1$-bounded.
		\item $f_1$ is nonnegative, has mean $\int f_1~d\mu = \int f~d\mu$, and obeys the bound
		\[f_1(x+r) = f_1(x) + O(\epsilon)\]
		whenever $x \in G$ and $r\in B(T,\gamma_1)$. 
		\item $\|f_2\|_{L^2(G)}\le\epsilon$.
		\item $\|\wh{f}_3\|_{\ell^\infty(\wh{G})}\le\gamma_2$.
	\end{enumerate}
\end{enumerate}
\end{proposition}

At this point, Tao \cite{TaoBlog} studies progressions with common difference in $B(T,\gamma_1)$. We require a small modification to handle more general matrix patterns. Define
\begin{equation*}
B' = \{r \in G: M_1r, M_2r \in B(T, \gamma_1)\}
\end{equation*}
Observe that $B'$ is the Bohr set $B(T', \gamma_1)$, where 
\begin{equation*}
T' = \{\xi \circ M_1: \xi \in T\} \cup \{\xi \circ M_2 : \xi \in T\}
\end{equation*}
The fact that $\xi \circ M_1$ and $\xi \circ M_2$ are elements of $\wh{G}$ is immediate by the identification $\wh{G} = \text{Hom}(G, \mb{R}/\mb{Z})$ and the fact that $M_1$ and $M_2$ are automorphisms. As an immediate consequence, we have for any $x \in G, d \in B'$ that
\begin{equation}\label{eq:Bohr_set_intersection}
	f_1(x+M_1d), f_1(x+M_2d) = f_1(x) + O(\epsilon).
\end{equation}

With this additional Lipschitz condition in the directions of $M_1$ and $M_2$ within $B'$, the remainder of the proof follows in the standard manner. We will now evaluate $$\int f(x)f(x+M_1d)f(x+M_2d)\mu_{B'} \ast \mu_{B'}(d)\mu(x).$$
Decompose each occurrence of $f$ into $f_1+f_2+f_3$, so that the integral has 27 terms. We will show that all terms other than the $f_1,f_1,f_1$ term are bounded in magnitude by $O(\epsilon)$.

We first bound terms containing $f_2$; we consider the representative case where the second term is $f_2$. 
In this case, take absolute values, and use $|f_a|,|f_b|\le 1$. We have
\[\int f_a(x)f_2(x+M_1d)f_b(x+M_2d)\mu_{B'} \ast \mu_{B'}(d)\mu(x)\le \int \mu_{B'} \ast \mu_{B'}(d) \int |f_2(x+M_1d)|\mu(x).\]
The inner integral is bounded by $\|f_2\|_{L^1} \le \|f_2\|_{L^2} \le\epsilon$. Thus we are left with a bound of $O(\epsilon)$ as desired.

We next bound the terms containing $f_3$. For the sake of simplicity we consider the case where the second term is $f_3$.
Standard Fourier analysis allows us to obtain that 
\begin{align*}
\int &f_a(x)f_3(x+M_1d)f_b(x+M_2d)\mu_{B'} \ast \mu_{B'}(d)\mu(x)\\
&= \sum_{\xi_4} |\widehat{\mu}_{B'}|^2(\xi_4)\sum_{\xi_1M_1+\xi_3(M_1-M_2)=\xi_4}\wh{f}_a(\xi_1) \wh{f}_3(-\xi_1-\xi_3) \wh{f}_b(\xi_3)
\end{align*}
We take absolute values and use $|\wh{f}_3|\le\gamma_2$. When $M_1$ and $M_1 - M_2$ are invertible, Cauchy--Schwarz lets us bound the inner sum by 
\begin{equation*}
\gamma_2\sum_{\xi_1} |\wh{f}_a(M_1^{-1} \xi_1)|^2\sum_{\xi_3} |\wh{f}_b((M_1 - M_2)^{-1}\xi_3)|^2\le\gamma_2
\end{equation*}
where the final inequality follows by Plancherel. Plancherel also implies $\sum_{\xi_4} |\widehat{\mu}_{B'}|^2 = \|\mu_{B'}\|_{L^2}^2 \le \|\mu_{B'}\|_{L^\infty} = O_{|T'|,\gamma_1^{-1}}(1).$ Recalling $|T'| \le 2|T|$ and that $\gamma_1$ is small with respect to $|T|$, we find that
\[\sum_{\xi_4}|\widehat{\mu}_{B'}|^2 = O_{\gamma_1^{-1}}(1).\]
The total contribution is therefore bounded by $\gamma_2\cdot O_{\gamma_1^{-1}}(1)\le\epsilon$ as long as $\omega_1,\omega_2$ grow fast enough.

Finally we are left with the term
\[\int f_1(x)f_1(x+M_1d)f_1(x+M_2d)\mu_{B'} * \mu_{B'}(d)\mu(x).\]
Since $\mu_{B'}\ast\mu_{B'}(d)$ is supported on $B' + B'$, for any $d$ in this support we have
\[f_1(x+M_id) = f_1(x)+O(2\epsilon)\]
for $i\in\{1,2\}$ by triangle inequality and two applications of (\ref{eq:Bohr_set_intersection}) each. Rewriting, and recalling $f_1$ is $[0,1]$-valued, we have
\begin{align*}
\int f_1(x)f_1(x+M_1d)f_1(x+M_2d)&\mu_{B'} \ast \mu_{B'}(d)\mu(x) \\
&= \int f_1^3(x)\mu_{B'} \ast \mu_{B'}(d)\mu(x) + O(\epsilon)\\
&\ge \bigg(\int f_1(x) \mu(x) \bigg)^3+O(\epsilon) &\text{(H\"older's inequality)}\\
&\ge \alpha^3 +O(\epsilon).
\end{align*} 
This completes the proof. \qed

We now deduce a pair of corollaries of this result.
\begin{theorem}\label{thm:3ptsFinite}
For any $\alpha,\epsilon > 0$ there exists $N_0$ so that the following holds. Let $G$ be a finite abelian group of order $N \ge N_0$. Let $M_1, M_2$ be automorphisms of $G$ so that $M_1-M_2$ is an automorphism. Then for any $A \subseteq G$ with $|A| \ge \alpha N$, there is a popular difference $d \neq 0$ so that
\begin{equation*}
\#\{x \in G : x, x+M_1d, x+M_2d \in A\}\ge (\alpha^3 - \epsilon) N.
\end{equation*}
\end{theorem}
\begin{proof}
Choose $N_0 = C/\epsilon$, where $C=C(\epsilon)$ is the constant from \cref{thm:3ptPatterns}. Given this choice, the contribution from $d = 0$ in (\ref{eq:3ptPatterns}) is at most $\epsilon N$, and therefore by the pigeonhole principle, some nonzero $d$ in the support of $\mu_D$ must have at least $\alpha^3 - O(\epsilon)$ density of this three-point pattern. This gives the desired result with an adjusted value of $\epsilon$.
\end{proof}

Finally we deduce the statement over the interval $[N]$ which was stated in the introduction.

\begin{proof}[Proof of \cref{thm:3ptsZ}]
Choose $N_0 = C/\epsilon$, where $C=C(\epsilon, |S_0| = 2k, \delta_0 = \epsilon/(2k))$ is the constant from \cref{thm:3ptPatterns}. Embed $S \subseteq[p]^k \hookrightarrow (\mb{Z}/p\mb{Z})^k$ naturally, where $p$ is a prime with $N < p < (1+\epsilon/k)N$. Initialize $\delta_0 = \epsilon/(2k)$, and
\[S_0 = \left\{x \mapsto \frac{(M_1x)_i}{p}\right\}_{i\in[k]} \cup \left\{x \mapsto \frac{(M_2x)_i}{p}\right\}_{i\in[k]}.\]

Given this setup, we apply \cref{thm:3ptPatterns}. Our condition ensures that our matrices have nonzero determinant (and are therefore invertible) modulo $p$. The contribution from $d = 0$ is at most $\epsilon$ as before, and we can find at least $(\alpha^3 - O(\epsilon))N^k$ patterns with common difference $d \neq 0$ for some $d \in B(S,\delta)+B(S,\delta) \subseteq B(S,2\delta)$, where $S \supseteq S_0$ and $\delta \le \delta_0$. (Note that the density may decrease by $O(\epsilon)$ under this embedding.) Plugging in the elements of $S_0$ and using $\delta\le\epsilon/(2k)$ gives us the constraints
\[(M_1x)_i, (M_2x)_i \in [-\epsilon p/k, \epsilon p/k], \text{ for all } i \in [k], x = (x_i)_{i = 1}^k \in B(S,2\delta).\]
Now we attempt to lift these $(\alpha^3 - \epsilon)N^k$ patterns into $\mb{Z}$ by viewing $d$ and each choice of $x$ as integer vectors. Throw out at most $\epsilon N^k$ choices of $x$ for which $x_i \notin [(\epsilon/k) p,(1-\epsilon/k)p]$ for some $i$. For the remainder, we see by triangle inequality in each coordinate that $x, x+M_1d, x+M_2d \in [p]^k$. By assumption, this triple must map to a three-point pattern under our embedding $[p]^k \hookrightarrow (\mb{Z}/p\mb{Z})^k$, so each of these points must have been an element of $A$ originally, and the triple forms a three-point pattern over $\mb{Z}$. We have therefore found $(\alpha^3-O(\epsilon))N^k$ of the necessary $3$-point patterns in $[N]^k$ as desired, which completes the proof upon adjusting $\epsilon$. 
\end{proof}
\begin{remark}
As usual, one can actually find $\Omega_{\alpha,\epsilon}(N^k)$ differences by using Markov's inequality and adjusting $\epsilon$.
\end{remark}

% \bibliographystyle{amsplain0.bst}
% \bibliography{main.bib}

\providecommand{\bysame}{\leavevmode\hbox to3em{\hrulefill}\thinspace}
\providecommand{\MR}{\relax\ifhmode\unskip\space\fi MR }
\providecommand{\MRhref}[2]{%
  \href{http://www.ams.org/mathscinet-getitem?mr=#1}{#2}
}
\providecommand{\href}[2]{#2}

\end{document}